\documentclass[journal,twoside,web]{ieeecolor}
  \usepackage{generic}
\usepackage{cite}
\usepackage{amsmath,amssymb,amsfonts}
\usepackage{algorithm}

 \usepackage{algorithm}

  \usepackage[noend]{algpseudocode}
 \usepackage{graphicx}
 \usepackage{textcomp}
\usepackage{bbm}

 \let\labelindent\relax
\usepackage{enumitem}

 \newtheorem{prob}{Problem}[section]
 \newtheorem{ass}{Assumption}[section]
 \newtheorem{prop}{Proposition}[section]
 \newtheorem{lem}{Lemma}[section]
 \newtheorem{thm}{Theorem}[section]
 \newtheorem{coro}{Corollary}[section]
 
 \newtheorem{defn}{Definition}[section]
 \newtheorem{rem}{Remark}[section]
 \newtheorem{exam}{Example}[section]

\makeatletter
\def\fnum@figure{\textcolor{subsectioncolor}{\sf Fig.~\thefigure}}
\def\fnum@table{\textcolor{subsectioncolor}{\sf TABLE~\thetable}}
\makeatother

\usepackage{color}
\makeatletter
\def\zz{\edef\zzz{\pdfliteral{\current@color}}}
\makeatother

\def\BibTeX{{\rm B\kern-.05em{\sc i\kern-.025em b}\kern-.08em
    T\kern-.1667em\lower.7ex\hbox{E}\kern-.125emX}}
\begin{document}
\title{Bang-Ride Optimal Control: Monotonicity, External Positivity, and Fast Battery Charging}
\author{Shengling Shi, \IEEEmembership{Member, IEEE},  Jacob Sass,
Jiaen Wu, Minsu Kim, Yingjie Ma, Sungho Shin, Rolf Findeisen, \IEEEmembership{Senior Member, IEEE}  Richard D. Braatz, \IEEEmembership{Fellow, IEEE}  
\thanks{Shengling Shi was with the Department of Chemical Engineering, Massachusetts Institute of Technology, USA, and is now with the Delft Center for Systems and Control, Delft University of Technology, the Netherlands (e-mail: s.shi-3@tudelft.nl). Jacob Sass, Minsu Kim, Yingjie Ma, Sungho Shin, and Richard D. Braatz are with the Department of Chemical Engineering, Massachusetts Institute of Technology, USA (e-mail: sassj@mit.edu; mskim77@mit.edu; yingjma@mit.edu; sushin@mit.edu; braatz@mit.edu). Jiaen Wu is with the Department of Mechanical Engineering, Stanford University, USA (e-mail: jiaenwu@stanford.edu), Rolf Findeisen is with the Department of Electrical Engineering and Information Technology, Technical University of Darmstadt (e-mail: rolf.findeisen@tu-darmstadt.de). 
 }  }

\maketitle 

\begin{abstract}
This work studies a class of optimal control problems with scalar inputs and general constraints, whose solutions follow a bang-ride pattern that always activates a constraint and enables efficient numerical computation. As a motivating example, fast battery charging leads to computationally demanding optimal control problems when detailed electrochemical models are used. Recently proposed optimization-free heuristics reduce this computational cost while producing input profiles observed in practice, following a bang-ride pattern and applying the maximum feasible input. We investigate when such heuristics satisfy necessary optimality conditions. By leveraging Pontryagin’s maximum principle, we unify and formalize existing insights on the bang-ride structure and on the optimal control attaining the maximum feasible input under monotonicity. We further establish a novel connection between the structured optimal control and the external positivity of the costate dynamics. These results provide a rigorous theoretical foundation for heuristic charging strategies and explain the efficiency of optimization-free algorithms.
\end{abstract}


\section{INTRODUCTION} \label{sec:Intro}
Many real-world control tasks can be formulated as optimal control problems with a single input, such as charging batteries \cite{berliner2022novel} or rapid thawing of frozen cells for cell therapy \cite{srisuma2024simulation}. However, solving the associated optimal control problems is often computationally demanding, requires significant resources and renders real-time implementation on limited platforms infeasible, despite recent advances in numerical optimal control \cite{andersson2019casadi}. This work considers optimal control problems whose structure allows one to obtain optimal or near-optimal solutions efficiently, without heavy online optimization. Under appropriate structural properties such as bang-ride behavior or monotonicity, the optimal input can often be determined by simple rules (e.g., applying the maximum feasible input) or validated a posteriori by simulation, enabling solutions that are orders of magnitude faster to compute than with conventional solvers. In particular, we take the fast battery charging problem as a motivating example, as the associated control profiles have these structural properties.

Charging batteries as quickly as possible without violating safety constraints and degrading battery health is crucial for a wide range of applications, such as electric cars and power storage. In practice, many input current profiles for fast charging are heuristic, based on the following concept \cite{li2021electrochemical}: They first apply the maximum current, and then the current is adjusted to maintain a series of active constraints. For example, the constant-current-constant-voltage (CC-CV) profile first applies the maximum current until the maximum allowed voltage is reached. The current is then adjusted to maintain the maximum voltage until the charging task is complete. Different charging profiles differ in the set of active constraints.  

Model-based optimal control offers an alternative for designing more advanced, degradation-aware charging profiles \cite{berliner2022novel}. Model-based control methods can optimally balance charging speed, safety, and battery lifetime by incorporating a detailed battery model and operation constraints. The bottleneck is the required resources to solve the resulting optimal control problems due to complex electrochemical battery models and nonlinear constraints. For example, with the implementation in \cite{galuppini2024efficient}, solving optimal charging problems using the gradient descent algorithm is estimated to take between $6$ and $80$ years using one Intel Core i7 CPU \cite{galuppini2024efficient}.

To address the computational bottleneck, a heuristic and optimization-free algorithm, called the \textit{hybrid simulation}, is proposed in \cite{berliner2022novel} to provide an approximate solution to the optimal control problem. Inspired by heuristic charging profiles in practice, the hybrid simulation approach to battery charging \cite{berliner2022novel} first applies the maximum current as the active (input) constraint. This active constraint and the battery model form a system of differential algebraic equations (DAEs) that are simulated. Then, whenever a new constraint becomes active during the simulation, such as an upper bound on the voltage, the active constraint is replaced by the new constraint to continue the simulation. This approach uses a single forward simulation to replace optimization, significantly reducing computational time. For example, it provides approximate solutions orders of magnitude faster than applying optimization solvers \cite{galuppini2024efficient}. However, the approximate solutions obtained can be suboptimal in some settings \cite{berliner2022novel}.
  
We study several theoretical problems motivated by the heuristic battery charging profiles in practice and hybrid simulation as discussed above. First, many charging profiles, including the solution of the hybrid simulation, activate a constraint at every time instant, either the maximum input bound or another constraint. This type of solution has been called \textit{bang-ride} control \cite{park2020optimal,taghavian2023selector}, a generalization of the classical \textit{bang-bang} control \cite{sussmann1979bang} that activates only the input bound. Second, while a general bang-ride control may switch between minimum and maximum input, the charging profiles and the hybrid simulation often apply the maximum feasible input. Finally, the answers to the above theoretical questions should lead to an understanding of when the practical charging profiles and the hybrid simulation are optimal. In summary this leads to the following core questions considered in this work:
\begin{itemize}
    \item \textit{Under what conditions does an optimal control problem lead to bang-ride control?}

    \item \textit{Under what conditions does the optimal control problem prefer the maximum input that is feasible?}

    \item \textit{When does the hybrid simulation satisfy the necessary optimality conditions?}
\end{itemize}

In the above questions, the optimal control in the second question is a special type of the general bang-ride control in the first question. In the third question, we focus on necessary optimality conditions, instead of sufficient conditions, following Pontryagin’s maximum principle (PMP).

\textit{Related Work:} The fast-charging problem has motivated several works to address some of the theoretical questions \cite{park2020optimal,taghavian2023selector,matschek2023necessary,drummond2023constrained}.  Bang-ride control has also been studied in process control \cite{benthack1997feedback,srinivasan2003dynamic}. The results in \cite{matschek2023necessary,benthack1997feedback,srinivasan2003dynamic} are related only to the first theoretical question, i.e., bang-ride control, generalizing the classical bang-bang theory \cite{sussmann1979bang,palanki1993synthesis}. However, (pure) state constraints, which do not depend on the input, are not considered in \cite{benthack1997feedback}, and no proof is provided in \cite{srinivasan2003dynamic}. It is well known in optimal control theory \cite{seierstad1986optimal,hartl1995survey} that state constraints require special treatment, compared to mixed constraints that depend on both the input and the state.

The result in \cite{park2020optimal} derives the optimal control analytically for a particular optimal charging problem with a linear constraint and shows that the resulting input is bang-ride. However, \cite{park2020optimal} does not provide a general theory for a wide class of control problems. Ref. \cite{taghavian2023selector} addresses the second question and observes that the monotonicity of the system, the objective, and the constraints play important roles. In particular, if the system is monotonic \cite{angeli2003monotone}, that is, a larger input leads to a larger state, and if an objective is minimized and is monotonically decreasing in the state, a maximum input is optimal when there is only the input bound. However, the effect of general constraints is more difficult to capture. The result in \cite{taghavian2023selector} shows that the maximum input remains optimal if the nonlinear constraints are increasing in the input and decreasing in the state, that is, applying larger state and input does not violate the constraints. However, our first question related to general bang-ride control is not considered in \cite{taghavian2023selector}, and many practical constraints are not decreasing in the state. Finally, \cite{drummond2023constrained} also studies bang-ride control and observes the importance of the monotonicity of the system in battery charging problems. 

\textit{Contribution:}
We focus on the three theoretical questions described above and demonstrate that PMP provides a unified framework for their analysis. To achieve bang-ride optimal control, while classical results show it is closely connected to nonlinear controllability \cite{benthack1997feedback,srinivasan2003dynamic}, we incorporate more general constraints and formalize the associated regularity condition on constraint switching, by exploiting classical regularity conditions in optimal control theory \cite{seierstad1986optimal}. In particular, the regularity condition requires an input discontinuity when a new state constraint becomes active, which is indeed observed in the application \cite{park2020optimal}, e.g., where the optimal current jumps when a state constraint is activated.  

For the second question, we establish a result analogous to the ones in \cite{taghavian2023selector, drummond2023constrained} but take a different approach via PMP. The result shows that the optimal control has the structure of the maximum feasible input under monotonicity of the system, objective, and constraints. While the theoretical insight is analogous to \cite{taghavian2023selector, drummond2023constrained}, exploiting PMP allows us to formalize regularity conditions, make an explicit connection to bang-ride control, and make further generalizations. 

In particular, the PMP-based approach provides a novel perspective: the structured input is optimal when the dynamic system, composed of the costate dynamics from PMP with the switching function as output, exhibits \textit{external positivity} \cite{drummond2023externally,weller2023external}, i.e., produces positive outputs. When specializing to monotone linear systems, we show that the structured optimal input and external positivity can be achieved by restricting the \textit{relative sensitivity} of the constraints, i.e., the sensitivity of a constraint to the state relative to its sensitivity to the input.

Combining the above results allows us to address the third question. We show the application of the theoretical results to a fast-charging battery problem, formulated for the single-particle battery model \cite{forman2010reduction} and a nonlinear voltage constraint.

\textit{Outline:} This article is structured as follows. Preliminaries and the problem are introduced in Section~\ref{sec:pre}. Bang-ride control is studied in Section~\ref{sec:Bang}. The last two questions are addressed in Section~\ref{sec:Mono}, and external positivity is studied in Section~\ref{sec:MonoGene}. The theory is applied to a battery charging problem in Section~\ref{sec:Apply}. All proofs are presented in the Appendix.

\textit{Notation:} For any matrix $M \in \mathbb{R}^{n\times m}$, $M \geqslant 0$ and $M >0$ denote element-wise inequality, i.e., $M$ only contains non-negative and positive elements, respectively. The symbols $1_{n\times m}$ and $0_{n \times m}$ denote $n\times m$ matrices with only ones and zeros, respectively. The notation $\mathrm{diag}(a_1,a_2,\dots,a_n)$ is a diagonal matrix with diagonal elements $a_i$. The optimal trajectory of the optimal control problem is denoted by $(x^\star,u^\star)$. For any function $s(x(t),u(t))$, we define the shorthand notation $s[t]:=s(x(t),u(t))$ and $s^\star[t]: = s(x^\star(t),u^\star(t))$. For brevity, the partial derivatives $\partial s /\partial x$ and $\partial s /\partial u$ are denoted as $s_x$ and $s_u$, respectively. A square real matrix is a \textit{Metzler matrix} if all its off-diagonal elements are non-negative.

\section{Preliminaries and Problem Formulation} \label{sec:pre}
\subsection{Nonlinear Systems}
 Following \cite{nijmeijer1990nonlinear}, consider the prototypical nonlinear system:
\begin{equation} \label{eq:sys}
\dot{x}(t) = F(x(t),u(t)) :=  f(x(t)) + g(x(t))u(t),
\end{equation}
where $f: \mathbb{R}^n \to \mathbb{R}^n$ and $g: \mathbb{R}^n \to \mathbb{R}^{n}$. Model~\eqref{eq:sys} can be used in various applications, e.g., mechanical systems \cite{nijmeijer1990nonlinear} and batteries \cite{liu2016extended}. Consider a piecewise continuous input $u: [0,t_f] \to  \mathbb{R}$ that is continuous at $0$ and $t_f$ and has a finite number of discontinuities. The value of $u$ at a point of discontinuity $t \in (0,t_f)$ is equal to its left-hand limit at $t$. We focus on a scalar input in this work, as this situation covers the application of interest. The generalization of the results to multiple inputs is out of the scope of this work.

The derivation of system-theoretical properties of \eqref{eq:sys}, such as nonlinear generalizations of controllability/observability, typically involves the operator known as the Lie bracket \cite{nijmeijer1990nonlinear}. Given two functions $f, \bar{f}: \mathbb{R}^n \to \mathbb{R}^n$ in class $C^\infty$, the Lie bracket is defined as 
$$
[f,\bar{f}](x):= \frac{\partial \bar{f}(x)}{\partial x }f(x) -  \frac{\partial f(x)}{\partial x }\bar{f}(x).
$$
For \eqref{eq:sys}, define $k=1,2,\dots$, \begin{align*}
 \mathrm{ad}^k_f g := [f, \mathrm{ad}^{k-1}_f g], \text{ and }  \mathrm{ad}_f^0 g: =g.
\end{align*}

\subsection{Optimal Control Theory}
Consider the optimal control problem: 
\begin{align}
 \min_{u,(t_f)} \quad & \phi(x(t_f))  + \int_{0}^{t_f} l(x(t)) dt, \label{eq:OC}\\
 \text{s.t. } & \dot{x}  = f(x) + g(x)u,\nonumber \\
 & s(x(t),u(t)) \leqslant 0,  \nonumber\\
 & h(x(t)) \leqslant 0,  \nonumber \\
 & z(x(t_f)) \geqslant 0, \text{ }x(0)=x_0, \nonumber
\end{align}
where the terminal time $t_f$ can be fixed or a free optimization variable, which will be specified when necessary. The functions $s: \mathbb{R}^n\times\mathbb{R} \to \mathbb{R}^{c_s}$ and $h:\mathbb{R}^n\to \mathbb{R}^{c_h}$ denote the mixed (state input) constraints and the (pure) state constraints, respectively. We assume that $s$ includes an input bound $u \in [u_{\mathrm{min}}, u_{\mathrm{max}}]$, and the remaining constraints are called the \textit{path constraints}. The functions $l, \phi: \mathbb{R}^n \to \mathbb{R} $ and $z: \mathbb{R}^n\to \mathbb{R}^{c_z}$ are the stage cost, the terminal cost, and the terminal constraint. Cost $l$ is independent of $u$ in this work.

PMP provides necessary conditions for optimal solutions of \eqref{eq:OC}. In PMP, regularity conditions on the time derivatives of state constraints are typically imposed \cite{hartl1995survey}. Consider the first-order time derivative\footnote{When $\dot{h} \equiv 0$, it is necessary to consider higher order time derivatives, see \cite[Thm.\ 6.1]{hartl1995survey}. To focus on the main result of this work, this technical complexity is avoided in this work.} of $h$, 
\begin{equation}
 \dot{h}(x,u) = \mathbf{J}_h(x) F(x,u),    
\end{equation}
where $\mathbf{J}_h $ is the Jacobian matrix of $h$. Combine $s$ and $\dot{h}$ by defining 
\begin{equation} \label{eq:GeneralConstraints}
\bar{s}_i = \begin{cases}
 s_i, \text{ }i =1,\dots,c_s,\\
 \dot{h}_{i-c_s}, \text{}i=c_s+1,\dots,c_s+c_h,
\end{cases}
\end{equation}
where each $\bar{s}_i$ depends on $u$. For brevity, $\bar{s}_i$ is said to be active for $i\in \{c_s+1,\dots, c_s+c_h\}$ if and only if (iff) the associated state constraint $h_{i-c_s}$ is active.

\begin{ass}[Standard regularity]  \label{ass:Regu}
\begin{enumerate}[label=(\alph*)]
    \item  The functions $f$, $g$, $\phi$, $s$, $l$, and $z$ are continuously differentiable, and each component of $h$ is in $C^2$.

    \item  For the optimal trajectory $(x^\star,u^\star)$ of \eqref{eq:OC}, the matrix
 $$
 \begin{bmatrix}
   \dot{h}^\star_u[t] & \mathrm{diag}(\dot{h}^\star[t]) & 0\\
s^\star_u[t] & 0 & \mathrm{diag}(s^\star[t])
 \end{bmatrix}
 $$
 has rank $c_h+ c_s$ almost every (a.e.) $t \in [0,t_f]$.
\end{enumerate}
 \end{ass}

Assumption~\ref{ass:Regu}b is the standard constraint qualification \cite{hartl1995survey}, analogous to constraint qualifications in optimization theory. It implies that at most one constraint is active at any time, as $u(t)$ is a scalar. A time instant $t$ is called an \textit{entry time} if some state constraint $h_i$ is inactive for a (non-zero) time interval before $t$ and becomes active for a time interval starting at $t$. 

The central objects of PMP are the Hamiltonian 
$$
H(x,u,\lambda_0,\lambda) =\lambda_0 l(x) + \lambda^\top (t) F(x,u),
$$
and the Lagrangian 
\begin{equation} \label{eq:Lagrange}
L(x,u,\lambda_0,\lambda, \mu,\nu) = H + \mu^\top(t)s+ \nu^\top(t) \dot{h} ,
\end{equation}
where $\lambda_0$ is a constant and $\lambda:[0,t_f] \to \mathbb{R}^n$, $\mu:[0,t_f] \to \mathbb{R}^{c_s}$, $\nu:[0,t_f] \to \mathbb{R}^{c_h}$ are multiplier functions.  Moreover, define the state-dependent control set \cite{hartl1995survey} 
\begin{align}
 \mathcal{D}(x) =\{u\mid & s(x,u) \leqslant 0, \nonumber \\ &\dot{h}_i(x,u) \leqslant 0 \text{ if } h_i(x)=0, i=1,\dots,c_h\}.   \label{eq:ControlSet}
\end{align}

Then, under Assumption~\ref{ass:Regu}, PMP \cite[Thm. 5.1]{hartl1995survey} states that, given an optimal solution $(x^\star, u^\star)$ of \eqref{eq:OC}, there exist $\lambda_0 \leqslant 0$, functions\footnote{The signs of $\lambda_0$, $\mu$, $\nu$, $\gamma$, and $\eta$ differ from \cite{hartl1995survey} as we consider minimizing the objective and reversed inequality constraints. $\lambda$ is piecewise absolutely continuous, and $\mu$ and $\nu$ are piecewise continuous.} $\lambda$, $\mu$, $\nu$, $\alpha \in \mathbb{R}^{c_z}$, $\gamma\in \mathbb{R}^{c_h}$, and a vector $\eta(\tau_i) \in \mathbb{R}^{c_h}$ for each time $\tau_i$ of discontinuity of $\lambda$ such that, for any $t \in [0,t_f]$,
\begin{equation} \label{eq:NonTrivial}
(\lambda_0,\lambda(t),\mu(t),\nu(t),\alpha,\gamma,\eta(\tau_1),\eta(\tau_2),\cdots) \not=  0,  
\end{equation}
and the following holds for a.e. $t\in [0,t_f]$: 
\begin{subequations}
\begin{align}
& H(x^\star(t),u^\star(t),\lambda_0,\lambda(t))  = \max_{   u \in \mathcal{D}(x^\star(t))   } H(x^\star(t),u,\lambda_0,\lambda(t)), \nonumber \\
&\partial L^\star/\partial u = 0,  \label{PMP:Lderive}\\
&\dot{\lambda}^\top =- \partial L^\star/\partial x , \label{PMP:costate}\\
&\mu(t) \leqslant 0, \text{ } \mu^\top(t) s^\star[t] = 0, \\
&\nu(t) \leqslant 0, \text{ }\nu^\top(t) h^\star[t] = 0, 
\end{align}
\end{subequations} 
and the transversality conditions at the terminal time $t_f$:
\begin{subequations}\label{PMP:costate-boundary-full}
\begin{align} 
& \lambda^\top(t_f^-) = \lambda_0 \phi^\star_x[t_f]+ \alpha^\top z^\star_x[t_f]+\gamma^\top h^\star_x[t_f], \label{PMP:costate-boundary} \\
& \alpha \geqslant 0, \text{ } \gamma \leqslant 0, \text{ } \alpha^\top z^\star[t_f] =\gamma^\top h^\star[t_f] =0. \label{PMP:CS} 
\end{align}
\end{subequations} 
At each entry time\footnote{Some literature also considers potential discontinuities at contact time \cite{hartl1995survey}, i.e., time $t$ when a state constraint is active but is inactive before and after $t$. As this is not often encountered in the application, we assume that there is no contact time for simplicity.} $\tau_i$, $\lambda_j(\tau_i)$, for some $j\in \{1,\dots,n\}$, may have a discontinuity satisfying
\begin{subequations}
\begin{align}
&\lambda^\top(\tau^-_i) = \lambda^\top(\tau^+_i)+ \eta^\top(\tau_i)h_x^\star[\tau_i], \label{eq:PMPCostateJump}\\
&\eta(\tau_i) \leqslant 0, \text{ } \eta^\top(\tau_i)h^\star[\tau_i]=0,
\end{align}
\end{subequations} 
where the right limit $\tau_i^+$ is replaced by $t_f$ when $\tau_i=t_f$, and the left limit $\tau_i^-$ is replaced by $0$ when $\tau_i = 0$.

If $t_f$ is free, it also holds that
\begin{equation} 
H^\star[t] = 0 \text{, } \forall t \in [0,t_f]. \label{PMP:ConstantH} 
\end{equation} 

\begin{rem}
    The above necessary conditions hold under another weak regularity condition so that $\lambda$ is well-behaved, i.e., it does not have infinite discontinuities. This is assumed to hold throughout this work. This technical detail is discussed in \cite{seierstad1986optimal,hartl1995survey}.
\end{rem}

The $\lambda$ is called the \textit{costate} due to the differential equation~\eqref{PMP:costate}. Note from \eqref{eq:PMPCostateJump} that activating state constraints may cause discontinuities of $\lambda$. The constant $\lambda_0$, if non-zero, can be normalized\footnote{ $\lambda_0$ is sometimes normalized to be $1$ when maximizing the objective is considered instead \cite{kirk2004optimal}.} to $-1$ \cite{liberzon2011calculus}, and it can be zero in some applications \cite{boscain2021introduction}, where the stage cost $l$ does not affect the optimal solution. The condition \eqref{eq:NonTrivial} is called the \textit{nontriviality condition} and will play an important role in our analysis.

\subsection{Optimization-Free Hybrid Simulation}
The goal of fast battery charging is to control the current to reach a desired state-of-charge (SOC) as fast as possible while respecting constraints, e.g., voltage and temperature constraints. The targeted SOC is typically formulated as a terminal constraint or terminal objective. The hybrid simulation approach aims to provide an approximate solution to this problem. Instead of its numerical implementation \cite{berliner2022novel}, we focus on its main logic in Algorithm~1.
 
\begin{algorithm}
\caption{Hybrid Simulation}\label{alg:example}
\begin{algorithmic}[1]
    \State Given $x(0)$, satisfying that  $(x(0),u_{\mathrm{max}})$ does not activate any path constraint, and initialize $u(0)=u_{\mathrm{max}}$
\For{$t \in [0,t_f]$}    
     \If{some $\bar{s}_i$ is active by $(x(t),u(t))$}
        \State Apply $u(t)$ satisfying $0= \bar{s}_i(x(t),u(t))$.
    \Else
        \State Apply $u(t) = u_{\mathrm{max}}$.
  \EndIf
\EndFor
 \end{algorithmic}
\end{algorithm}
 
The hybrid simulation is stopped if the target is reached, e.g., a certain SOC. In practice, the time horizon is discretized, and the active constraint in Step $4$ and the system model lead to a DAE that is simulated by a DAE solver \cite{berliner2022novel}. Note that we have made an implicit but practical assumption: At any $t$, at most one constraint is active so that the solution is feasible, consistent with Assumption~\ref{ass:Regu}b. Incorporating different sets of constraints into the optimal control problem can lead to different solutions to the hybrid simulation \cite{berliner2022novel}.

\subsection{Problem Formulation}
Several practical charging profiles are consistent with the solution of the hybrid simulation. Our goal is to investigate these charging profiles and the hybrid simulation from a theoretical perspective via the three questions in Section~\ref{sec:Intro}. These charging profiles are examples of bang-ride control, i.e., they either apply $u_{\mathrm{max}}$ or $u_{\mathrm{min}}$ (\textit{bang}), or keep a path constraint active (\textit{ride}). 

\begin{defn}
An optimal trajectory $(x^\star,u^\star)$ of problem \eqref{eq:OC} is said to be \textit{bang-ride} if,  for a.e. $t \in [0,t_f]$, there exists an active constraint.\footnote{We regard the bang-bang control, i.e., $u(t) \in \{u_{\mathrm{min}}, u_{\mathrm{max}}\}$, as a special case of bang-ride control.}  
\end{defn}

Moreover, they are special classes of bang-ride control that prefer the maximum feasible input. As shown in Example~\ref{exam:1}, it is not clear whether such control actions are optimal even in simple cases.

 \begin{exam} \label{exam:1}
Consider the optimal control problem
 \begin{align}
 \min_u  \quad & -x_1(t_f) \label{eq:SpecialOpti2}\\
 \text{s.t. } & \begin{bmatrix}
   \dot{x}_1\\
   \dot{x}_2
 \end{bmatrix}  = \begin{bmatrix}
     a_1 & 0\\
     0 & a_2
 \end{bmatrix}\begin{bmatrix}
   x_1\\
   x_2
 \end{bmatrix}+\begin{bmatrix}
     1\\
     1
 \end{bmatrix} u, \text{ } x(0)=\begin{bmatrix}
     0.5\\
     0.5
 \end{bmatrix}  \nonumber \\ & u \in [0,1], \text{ } \begin{bmatrix}
   1 & 1  
 \end{bmatrix} \begin{bmatrix}
   x_1\\
   x_2
 \end{bmatrix} u-4 \leqslant 0. \label{eq:SpecialOptiCons}
\end{align}
The system shows that given a fixed initial state, applying a larger input leads to a larger state trajectory. Such system is called a monotone system \cite{angeli2003monotone} and is also a positive system \cite{de2001stabilization}. Then, increasing the state also decreases the objective $-x_1(t_f)$. Therefore, a reasonable guess is that the optimal input applies $u=u_{\mathrm{max}}= 1$ and then rides the path constraint~\eqref{eq:SpecialOptiCons} if active, i.e., consistent with the solution of the hybrid simulation. This is in fact optimal when $a_1=0$ and $a_2=-1$, as shown in Fig.~\ref{fig:Special}.
 \begin{figure}[h]
    \centering
    \includegraphics[width=0.5\textwidth]{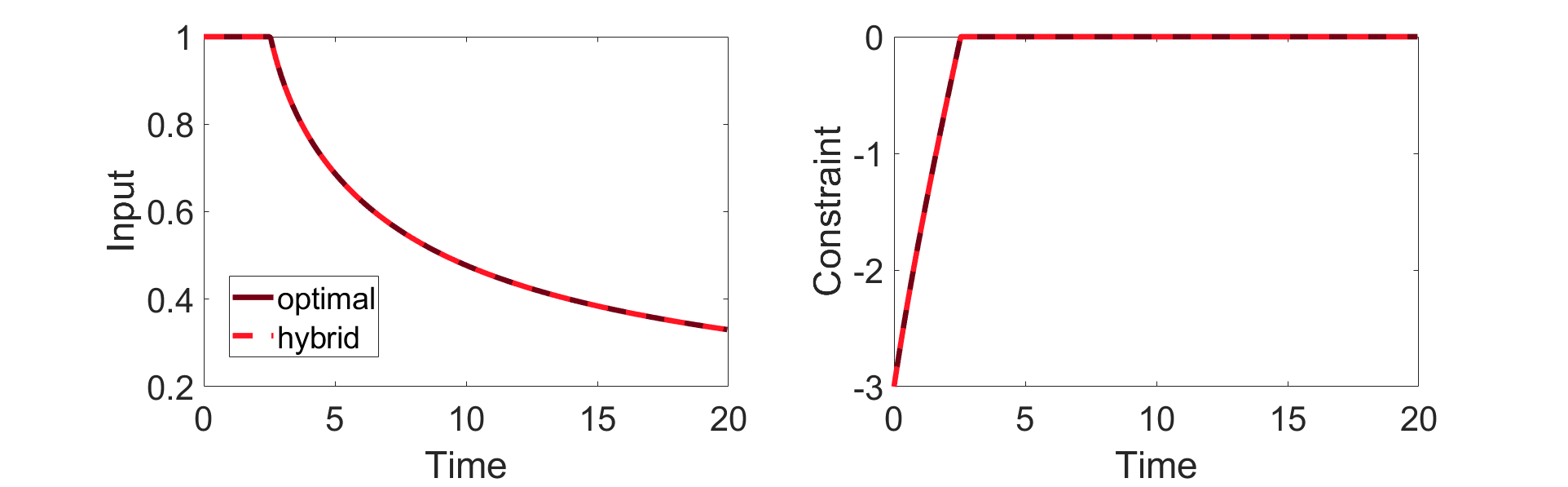}
    \caption{If $a_1=0$ and $a_2=-1$, the optimal input of \eqref{eq:SpecialOpti2} and the trajectory of the path constraint~\eqref{eq:SpecialOptiCons} are consistent with the hybrid simulation.}
    \label{fig:Special}
\end{figure}

 \begin{figure}[h]
    \centering
    \includegraphics[width=0.5\textwidth]{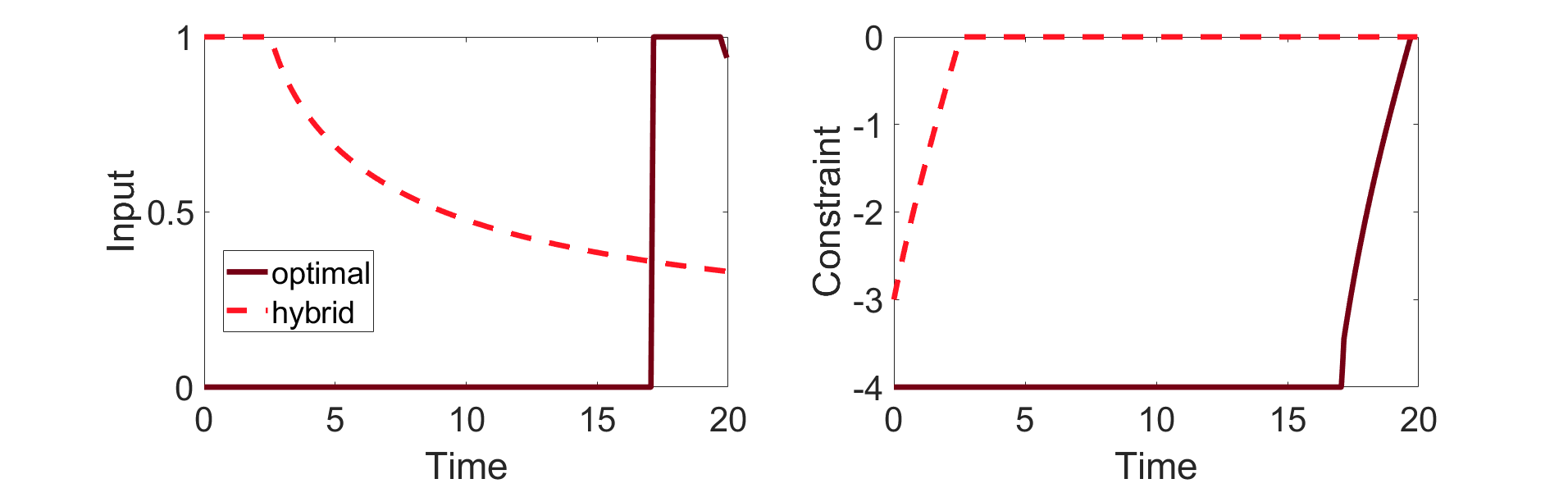}
    \caption{The optimal input of \eqref{eq:SpecialOpti2} with $a_1=-1$ and $a_2=0$ does not provide the maximum input for $t\in [0,15]$, even though the path constraint~\eqref{eq:SpecialOptiCons} is not active. The hybrid simulation is not optimal in this case. } 
    \label{fig:Special2}
\end{figure}
However, if $a_1=-1$ and $a_2=0$, where the system remains monotonic, the optimal input applies the minimum input instead for some time interval, as shown in Fig.~\ref{fig:Special2}. The hybrid simulation is not optimal and leads to the terminal objective $-0.34$, while the optimal control has the objective $-0.94$. \begin{figure}[h]
    \centering
    \includegraphics[width=0.5\textwidth]{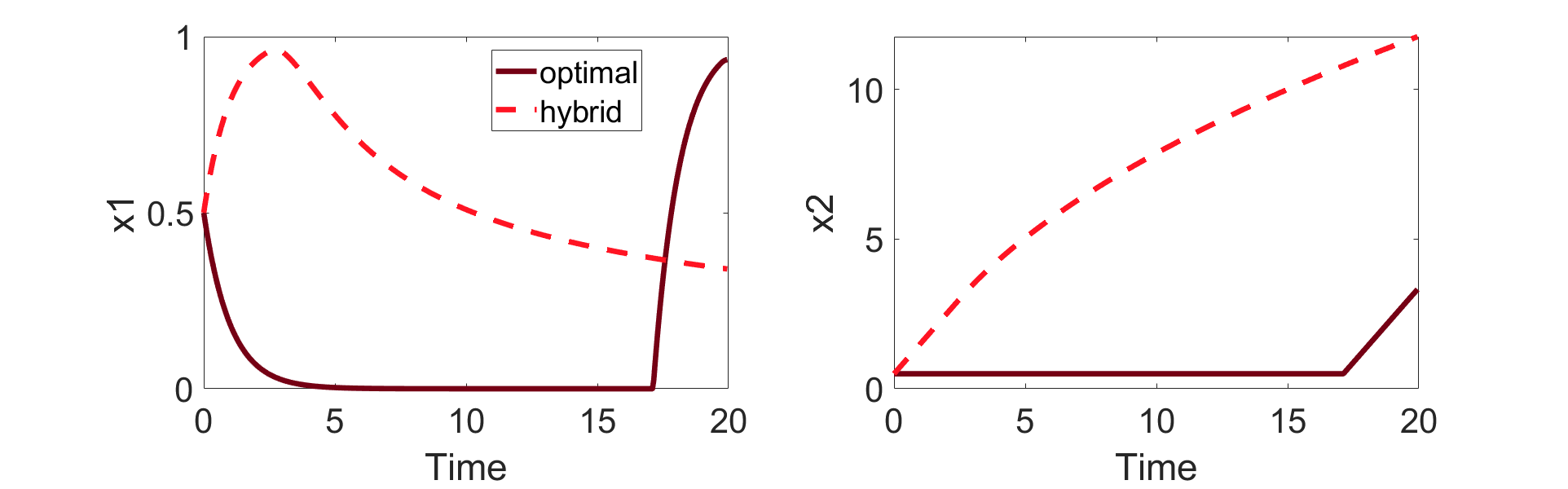}
    \caption{The state trajectories corresponding to Fig.~\ref{fig:Special2} are shown.} 
    \label{fig:Special2X}
\end{figure}
The associated state trajectories in Fig.~\ref{fig:Special2X} show that the hybrid simulation initially increases both states; however, the constraint soon becomes active and limits the input, decreasing $x_1$ and thus having a suboptimal objective. Instead, the optimal control increases the states later, avoiding early activation of the path constraint and leading to a larger $x_1$ at the terminal time. Therefore, this performance gap between the hybrid simulation and optimal control is due to the combined effect of system dynamics and the constraint. \hfill$\triangle$

\end{exam}
\section{Bang-Ride Control} \label{sec:Bang}
\subsection{Introduction and System Properties}
We address the first question and establish the conditions for general bang-ride control. Analogous to the bang-bang control \cite{palanki1993synthesis}, its condition~\eqref{eq:Swiching} on the so-called \textit{switching function} $\lambda^\top g$ also achieves the bang-ride control, as shown in Lemma~\ref{lem:BangRide}.

\begin{lem} \label{lem:BangRide}
For problem~\eqref{eq:OC}, if Assumption~\ref{ass:Regu} holds and if the optimal trajectory satisfies \begin{equation} \label{eq:Swiching}
\lambda^\top(t) g^\star[t] \not= 0 \text{  for a.e. } t\in [0,t_f],
\end{equation} then the optimal trajectory of \eqref{eq:OC} is bang-ride.   
\end{lem}
\begin{proof}
 The reasoning is classical and is presented for completeness. Eq.~\eqref{PMP:Lderive} shows that $\lambda^\top g^\star + \mu^\top(t)s^\star_u+ \nu^\top(t) \dot{h}^\star_u=0$. Then, at every $t$, $\lambda^\top(t) g^\star[t] \not=0$ leads to $\mu^\top(t)s^\star_u[t]+\nu^\top(t) \dot{h}^\star_u[t]\not=0$. There must be an active constraint at $t$; otherwise, $\mu(t) = 0$ and $\nu(t)=0$.
\end{proof}

Condition~\eqref{eq:Swiching} can be guaranteed by
\begin{itemize}
    \item relevant system properties associated with $g$, and 
        \item $\lambda(t) \not=0 $ for a.e. $t$.
\end{itemize}

The system property is related to (nonlinear local) controllability, as shown later. The non-zero $\lambda$ shows that the optimal control problem is nontrivial, as the system dynamics always affect the optimal solution via the Lagrangian~\eqref{eq:Lagrange}.

This subsection briefly recalls the system property to achieve bang-bang control \cite{palanki1993synthesis}, which is also useful for bang-ride control. The main difference in establishing bang-ride control from bang-bang control lies in guaranteeing non-zero $\lambda$, which requires addressing the additional path constraints. This is presented in the next subsection.

A necessary condition for $\lambda^\top(t) g(t)=0$ over a time interval is that its time derivatives vanish, i.e.,
\begin{equation} \label{eq:HighDeri}
(\lambda^\top(t) g(t) )^{(i)} = 0, \text{ } i \in \mathbb{N},
\end{equation} 
where we assume:
\begin{ass} \label{ass:Cinf}
 The functions $f$ and $g$ are $C^\infty$-functions.   
\end{ass}
If the system has a special property such that \eqref{eq:HighDeri} is not allowed, then \eqref{eq:Swiching} holds. It is well-known that the high-order time derivatives in \eqref{eq:HighDeri} can be computed by Lie brackets, which leads to the properties in Assumption~\ref{ass:Control}:
\begin{ass}[\cite{palanki1993synthesis}] \label{ass:Control}
The optimal trajectory satisfies 
for any $x\in \left\{x^\star(t)\mid t \in [0,t_f] \right\}$,
\begin{enumerate}[label=(\alph*)]
    \item 
 $
 [g,\mathrm{ad}^{v-1}_f g](x) \in  \mathrm{span}\{g(x),\mathrm{ad}_f g(x),\dots,\mathrm{ad}^{v-1}_f g(x)\},
 $ for $v = 1, 2, \cdots$.

 \item $M_i(x) :=\begin{bmatrix}     
g(x) &\mathrm{ad}_f g(x) &\cdots &\mathrm{ad}^{i-1}_f g(x) \end{bmatrix}$ has\footnote{A stronger condition may be considered by letting $i=n$ \cite{benthack1997feedback}.} rank $n$ for some $i\geqslant n$,
\end{enumerate}
\end{ass}

\begin{rem}[Nonlinear Controllability]
Assumption~\ref{ass:Control} is related to controllability. For linear systems, $M_n$ is the controllability matrix, and thus Assumption~\ref{ass:Control} implies linear controllability. For nonlinear systems, Assumption~\ref{ass:Control} implies \textit{locally strong accessibility} \cite{nijmeijer1990nonlinear}, i.e., one version of local controllability. It also implies that the system is \textit{feedback linearizable} \cite{nijmeijer1990nonlinear} and \textit{differentially flat} \cite{beaver2024optimal}. More detailed discussions are presented in Appendix~A.
\end{rem}

\begin{lem} \label{lem:Reuglar}
For problem~\eqref{eq:OC}, suppose that Assumptions~\ref{ass:Regu} and \ref{ass:Cinf} hold. Then, if $\lambda(t) \not =0$ for a.e. $t\in [0,t_f]$, and if Assumption~\ref{ass:Control} holds, the optimal trajectory $(x^\star,u^\star)$ of problem~\eqref{eq:OC} is bang-ride.
\end{lem}

Lemma~\ref{lem:Reuglar} is a variant of \cite[Thm.\ 1]{palanki1993synthesis} and \cite[Thm.\ 6]{benthack1997feedback} following the same proof. The differences are that no constraints are considered and no state constraints are in \cite{palanki1993synthesis} and \cite{benthack1997feedback}, respectively. These simplified problems lead to $\lambda\not=0$ in a straightforward way. In contrast, due to state constraints in this work, $\lambda\not=0$ is explicitly assumed in Lemma~\ref{lem:Reuglar} and needs to be achieved through additional regularity conditions, as done in the next subsection.

\subsection{Qualification and Switching of Constraints} 
\label{sec:ConstraintQuali}
We investigate conditions for achieving a non-zero $\lambda$, as required in Lemma~\ref{lem:Reuglar}. We exploit a series of regularity conditions, including constraint qualifications and conditions on constraint switching. 
Constraint qualifications are typically rank conditions on the constraints. Regularity on constraint switching excludes potential discontinuities of $\lambda$ caused by the activation/deactivation of state constraints. Some regularity conditions in this subsection are classical, but have not been used to study the bang-ride property.

For the constraint qualification, the settings in Assumption~\ref{eq:Termi} generalize the classical setting in \cite{palanki1993synthesis}.
\begin{ass}[Terminal objective and constraint]\label{eq:Termi}
 One of the following conditions holds:
   \begin{enumerate}[label=(\alph*)]
     \item $t_f$ can be free or fixed, $l \equiv 0$, and 
    \begin{equation}      \label{eq:Arank2}
    \mathrm{rank}\!\left( \begin{bmatrix}
\partial \phi/\partial x & 0 \\
      \partial z/  \partial x  & \mathrm{diag}(z) 
    \end{bmatrix}   \Bigg\rvert_{x^\star(t_f)}\right)\! = c_{z}+1.
    \end{equation}
    
   \item  $t_f$ is free, $l(x^\star(t))\not=0$ for any $t \in [0,t_f]$, and
    \begin{equation}      \label{eq:Arank}
\mathrm{rank}\!\left(\begin{bmatrix}
      \partial z  /\partial x   & \mathrm{diag}(z) 
    \end{bmatrix}     \big\rvert_{x^\star(t_f)}\right)\! = c_{z}.
    \end{equation}

    \item $t_f$ is fixed, $l_x^\star[t] \not= 0$ for a.e. $t \in [0,t_f]$, and \eqref{eq:Arank} holds.
\end{enumerate}
\end{ass}

 When $t_f$ is free, a common choice is $l \equiv 1$, leading to the time-optimal control problem. Assumption~\ref{eq:Termi}b generalizes from $l \equiv 1$ to a non-zero objective function. Assumption~\ref{eq:Termi}a and c extend \cite{palanki1993synthesis} to incorporate the terminal constraint and a non-zero objective, respectively.\footnote{Equation \eqref{eq:Arank} is satisfied vacuously if $z$ is absent.}

When there is no path constraint, Assumption~\ref{eq:Termi} achieves non-zero $\lambda$, analogous to the classical result \cite{liberzon2011calculus}. In this work, potential discontinuities of $\lambda$ caused by state constraints need to be addressed. We exploit classical regularity conditions in \cite[Chap.\  5.4]{seierstad1986optimal} for our purposes. First, define the index set of active state constraints right after and right before time $t$:
\begin{align*}
\mathcal{G} ^\star(t^+ ) := \{ i\mid  &\exists \text{ }k>t, \text{arbitrarily near }t,  \text{such that} \\& h_i(x^\star(k))=0, i =1,\dots,c_h\}.\\
\mathcal{G}^\star(t^-) := \{ i\mid & \exists \text{ }k<t, \text{arbitrarily near }t, \text{such that} \\&h_i(x^\star(k))=0, i =1,\dots,c_h\}.
\end{align*}  
Further define $\mathcal{G}^\star(t) := \{i\mid h_i(x^\star(t)) =0,i=1,\dots,c_h\}$. Then, if a state constraint becomes active at $t$, $\mathcal{G}(t^-) \not= \mathcal{G}(t)$.

\begin{ass}[Switching of state constraints]  \label{ass:ContBoundary}
The following properties hold for the optimal trajectory\footnote{These properties hold vacuously if there is no pure state constraint.}:
\begin{enumerate}[label=(\alph*)]
    \item There is no switching of state constraints at time $0$ and $t_f$, i.e., $\mathcal{G}^\star(0) = \mathcal{G}^\star(0^+)$ and $\mathcal{G}^\star(t_f) = \mathcal{G}^\star(t_f^-)$.

    \item At each entry time $\tau_i \in (0,t_f)$ defined in \eqref{eq:PMPCostateJump}, $u^\star$ is discontinuous at $\tau_i$.
\end{enumerate}
\end{ass}

Assumption~\ref{ass:ContBoundary} ensures the continuity of $\lambda(t)$ on $[0,t_f]$ despite the switching of state constraints. In Assumption~\ref{ass:ContBoundary}a, the absence of constraint switching at time $0$ and $t_f$ is common in practical situations. Assumption~\ref{ass:ContBoundary}b regulates the constraint switching for $t \in (0,t_f)$. It means that, when a new state constraint becomes active, the input has a jump. Equivalently, Assumption~\ref{ass:ContBoundary}b requires a jump of the time derivative $\dot{h}^\star(\tau_i)$ of the state constraint, due to the jump of $u^\star$ and $\dot{h}^\star(\tau_i) = h_x(x^\star(\tau_i)) F(x^\star(\tau_i),u^\star(\tau_i))$. This is common in real-world applications. For example, a state (e.g., temperature) tends to increase with a positive time derivative, and then hits its upper bound and remains at this value, i.e., its time derivative jumps from positive to zero. This also leads to a jump in $u$ to keep this state at its upper bound.

Based on Assumption~\ref{ass:ContBoundary} and the nontriviality condition \eqref{eq:NonTrivial}, we show $\lambda\not=0$ for problem \eqref{eq:OC}.
 
\begin{prop}\label{prop:NonZeroLambda}
For  problem~\eqref{eq:OC} and its optimal trajectory, if Assumptions~\ref{ass:Regu}, \ref{eq:Termi}, and \ref{ass:ContBoundary} hold, $\lambda(t)\not=0$ for a.e. $t \in [0,t_f]$. Moreover, $\lambda(t) \not=0$ for any $t \in [0,t_f]$ under Assumptions~\ref{eq:Termi}ab.
\end{prop}

The proof of Proposition~\ref{prop:NonZeroLambda} is presented in the Appendix. In Section~\ref{sec:MonoGene}, we mainly use Assumptions~\ref{eq:Termi}a for $l \equiv 0$ and non-zero $\lambda$ for \textit{any} $t$.

\subsection{Main Result for Bang-Ride Control}
Building upon the previous results on non-zero $\lambda$ in Proposition~\ref{prop:NonZeroLambda} and the system properties in Lemma~\ref{lem:Reuglar}, the main theoretical result is obtained straightforwardly:
\begin{thm} \label{thm:Main}
If problem~\eqref{eq:OC} and its optimal trajectory satisfy Assumptions~\ref{ass:Regu}, \ref{ass:Cinf}, \ref{ass:Control}, \ref{eq:Termi}, and \ref{ass:ContBoundary}, the optimal trajectory $(x^\star,u^\star)$ is bang-ride.
 \end{thm}

Theorem~\ref{thm:Main} shows that combining the system property, related to controllability, and the regularity of constraint switching leads to bang-ride control. These conditions are sufficient. As shown in Sections~\ref{sec:Mono} and \ref{sec:MonoGene}, bang-ride control can also be achieved by monotone systems that are not controllable. However, the regularity of constraint switching remains important.

\subsection{System Decomposition} \label{sec:Subsystem}
Even if the complete system may not satisfy Assumption~\ref{ass:Control}, an important situation, pointed out in \cite{benthack1997feedback}, is when only a subpart of the system affects the optimal control solution and satisfies the assumption. One straightforward situation is when the system is in the form: 
\begin{subequations} \label{eq:SysSub}
\begin{align}
  \dot{x}_1& = f_1(x_1)+ g_1(x_1)u, \label{eq:sub1} \\
 \dot{x}_2 &= f_2(x_1,x_2)+ g_2(x_1,x_2)u, \label{eq:sub2}
\end{align}
\end{subequations}
and all the functions in problem~\eqref{eq:OC} depend only on $x_1$. Then, in addition to the regularity conditions in the constraints, the system properties in Assumption~\ref{ass:Control} only need to be satisfied by \eqref{eq:sub1} to achieve the bang-ride property.

A more interesting situation occurs when the system is not in the form of \eqref{eq:SysSub} but can be transformed to this form via a coordinate transformation. We illustrate this considering linear systems. Suppose that problem~\eqref{eq:OC} has the form:
\begin{align}
 \min_{u,(t_f)} \quad & \phi(y(t_f))  + \int_{0}^{t_f} l(y(t)) dt, \label{eq:OCdem}\\
 \text{s.t. } & \dot{x} = Ax+Bu, \text{ }y=Cx, \label{eq:LinearSys}\\
 & s(y(t),u(t)) \leqslant 0, \text{ }h(y(t)) \leqslant 0,  \nonumber \\
 &z(y(t_f)) \geqslant 0, \text{ }x(0)=x_0, \nonumber
\end{align}
where $C \in \mathbb{R}^{m \times n}$, e.g., $Cx$ is a sub-vector of $x$. The linear system~\eqref{eq:LinearSys} has an associated Kalman decomposition:
\begin{equation} \label{eq:KalDem}
\begin{bmatrix}
\dot{\bar{x}}_1 \\
 \dot{\bar{x}}_2
\end{bmatrix}= \begin{bmatrix}
 \bar{A}_{11} & 0 \\
   \bar{A}_{12} &  \bar{A}_{22}
\end{bmatrix}\begin{bmatrix}
\bar{x}_1 \\
\bar{x}_2
\end{bmatrix}+ \begin{bmatrix}
  \bar{B}_{1}\\
  \bar{B}_{2}
\end{bmatrix} u, \text{ } y = \bar{C} x_1.
\end{equation} 
As only $y$ affects the optimal solution,  solving problem~\eqref{eq:OCdem} using system model~\eqref{eq:KalDem} leads to the same optimal solution. Therefore, to ensure the bang-ride property, only the subsystem $\dot{ \bar{x}}_1 =  \bar{A}_{11}  \bar{x}_1 + \bar{B}_1 u$ needs to satisfy Assumption~\ref{ass:Control}, i.e., being controllable. 

The difficulty in applying the above reasoning to nonlinear systems is in finding a global coordinate transformation. Transformation by exploiting the observability codistribution is only defined locally \cite{nijmeijer1990nonlinear}.

\section{Monotonicity For Maximum Input} \label{sec:Mono}
\subsection{Introduction}
The previous theory covers the general bang-ride property; however, there are more structural properties in common battery charging profiles and the solution of the hybrid simulation:
\begin{itemize}
    \item  While a bang-ride controller may switch between the minimum and maximum input, common battery charging profiles and the hybrid simulation apply the maximum input.
 \end{itemize}

The above property relates to the following straightforward consequence of PMP:
\begin{lem}
Given problem~\eqref{eq:OC} and its optimal trajectory $(x^\star,u^\star)$, suppose that Assumption~\ref{ass:Regu} holds. If 
\begin{equation} \label{eq:No Swiching}
\lambda^\top(t) g^\star[t] > 0 \text{  for a.e. } t\in [0,t_f],
\end{equation}
then \begin{equation}    \label{eq:maxControl}
u^\star(t) = \max  \mathcal{D}(x^\star(t)), \text{ for a.e. } t\in [0,t_f],
\end{equation} with the set-valued function $ \mathcal{D}$ defined in \eqref{eq:ControlSet}.
\end{lem}

The positive $\lambda^\top g^\star$ in \eqref{eq:No Swiching} leads to a special class of bang-ride control, compared to being non-zero in \eqref{eq:Swiching}. When there is only an input constraint $u \in [u_{\mathrm{min}}, u_{\mathrm{max}}]$, \eqref{eq:No Swiching} and \eqref{eq:maxControl} lead to $u^\star=u_{\mathrm{max}}$. When other path constraints are present, \eqref{eq:maxControl} shows that the optimal input is the maximum feasible input. Our goal is to develop conditions for \eqref{eq:No Swiching}.

 \subsection{Limitation of Controllability}
As controllability and the regularity conditions establish $\lambda^\top g \not=0$ in Theorem~\ref{thm:Main}, a straightforward way to attain $\lambda^\top g >0$ is to extend Theorem~\ref{thm:Main}:
 \begin{coro} \label{coro:ControlMax}
   In the setting of Theorem~\ref{thm:Main}, if it holds additionally that for every $\bar{s}_i$ active for a non-zero time interval, $ \partial \bar{s}_i^\star[t] /\partial u>0$ holds for a.e. $t$ within its active time interval, then \eqref{eq:No Swiching} and \eqref{eq:maxControl} hold.
 \end{coro}

However, condition $ \partial \bar{s}_i^\star[t] /\partial u>0$ cannot be verified a priori as the optimal trajectory is unknown, an inherent problem due to PMP. In this case, other informative properties that can be analyzed a priori, such as system properties, are desired.  

Nevertheless, as both systems in Example~\ref{exam:1} are controllable, controllability is not informative enough to distinguish \eqref{eq:No Swiching} from the general bang-ride control. This motivates us to exploit monotonicity. As shown later in Example~\ref{exam:Analysis}, this direction leads to system properties that can distinguish the two systems in Example~\ref{exam:1}.

\subsection{Monotonicity}
Condition~\eqref{eq:No Swiching} can be guaranteed by one of the conditions:
\begin{itemize}
    \item $\lambda(t)>0$ and $g^\star[t]>0$; 

    \item $\lambda(t)<0$ and $g^\star[t]<0$; 
\end{itemize}
We consider the first case, and the second case follows analogously. As $g^\star[t]>0$ can be easily verified given a trajectory, the focus is on ensuring $\lambda>0$. This is stronger than requiring $\lambda\not=0$ as in Section~\ref{sec:ConstraintQuali}. 

First, imposing monotonicity on the terminal objective and constraint can ensure $\lambda(t_f)>0$. We collect the terminal constraint, the terminal objective, and their multipliers as
\begin{equation*}
\bar{z}_i= \begin{cases}
 z_i ,& i=1,\dots, c_z,\\
 \phi, & i=c_z+1.
\end{cases} \quad \bar{\alpha}_i = \begin{cases}
   \alpha_i, & i =1,\dots, c_z,\\
\lambda_0, & i=c_z+1.
\end{cases}    
\end{equation*}
Define $$\mathcal{I}_f = \{i = 1,\dots,c_z+1 \mid \bar{\alpha}_i \not=0\},$$
which contains the indices of the terminal constraint/objective that affect the optimal solution. Then, introduce the following assumptions on the terminal components:
\begin{ass}[Monotonicity in terminal components]\label{ass:MonoTf} Assume that $\partial \bar{z}^\star_i[t_f] / \partial x \leqslant 0$ for any $i \in \mathcal{I}_{f}$.
\end{ass}

Assumption~\ref{ass:MonoTf} requires the terminal components to decrease monotonically. Note that the inactive terminal constraints, whose associated multipliers are zero based on \eqref{PMP:CS}, are excluded from the above restriction via set $\mathcal{I}_f$. 

Minimizing a decreasing terminal objective $\phi$ leads to an incentive for a larger state, which can be achieved by applying a larger input for special system classes. This is related to monotonicity of the system.
\begin{ass}[Monotonicity in system dynamics] \label{ass:MonoSys}
The system satisfies the conditions
\begin{enumerate}[label=(\alph*)]
    \item $g^\star[t]>0$,
    \item $F_x^\star[t]$ is a Metzler matrix,
\end{enumerate}
along the optimal trajectory $(x^\star,u^\star)$ for a.e. $t\in [0,t_f]$.
\end{ass}

Assumption~\ref{ass:MonoSys}a is useful for establishing \eqref{eq:No Swiching}. Moreover, based on \cite[Prop. 3.3]{angeli2003monotone}, Assumption~\ref{ass:MonoSys} indicates that the system is a \textit{monotone system} \cite{angeli2003monotone} in a neighborhood of the optimal trajectory, i.e., a larger initial state and larger input lead to a larger state trajectory. Assumption~\ref{ass:MonoSys}b can also be guaranteed by the following more explicit conditions:
 \begin{coro}
Suppose that Assumption~\ref{ass:Regu} holds. If $f_x^\star[t]$ and $g_x^\star[t]$ are Metzler matrices and $u^\star(t) \geqslant 0$, then $F_x^\star[t]$ is a Metzler matrix.
 \end{coro}

Finally, the path constraints and path objective $l$ should behave well so that the optimal solution prefers a larger input. 
\begin{ass}[Monotonicity in path components] \label{ass:MonoPath} \hspace{1mm}
\begin{enumerate}[label=(\alph*)]
    \item $l^\star_x[t] \leqslant 0$ for a.e. $t \in [0,t_f]$.
    
    \item For every $\bar{s}_i$ active for a non-zero time interval, 
    \begin{equation} \label{eq:PathMono}
\frac{\partial \bar{s}_i^\star[t]}{ \partial u} > 0, \text{ } \frac{\partial \bar{s}_i^\star[t] }{ \partial x} \leqslant 0,    \end{equation}
hold for a.e. $t$ within its active time interval.
\end{enumerate}
\end{ass}

Assumption~\ref{ass:MonoPath}a requires $l$ to be monotonically decreasing, analogously to Assumption~\ref{ass:MonoTf}, and thus a larger state is desired to decrease the objective. Assumption~\ref{ass:MonoPath}b imposes that changing $u$ and $x$ have opposite effects on constraint $\bar{s}_i$. 

We can now obtain an intuition about how the properties in Assumptions~\ref{ass:MonoTf}, \ref{ass:MonoSys}, and \ref{ass:MonoPath} affect the optimal solution. The monotonicity of the terminal components and the stage cost shows that an optimal solution prefers a larger state. This leads to a preference for a larger input, due to the monotonicity of the system. Finally, the larger input and the larger state do not tend to violate the constraints, as they have the opposite effect on the constraint $\bar{s}_i$. Under these properties, the optimal input is likely to be the maximum feasible input.

Indeed, the above intuition can be formalized as follows:
\begin{thm}  \label{thm:mono}
Given problem~\eqref{eq:OC} and its optimal trajectory $(x^\star,u^\star)$, suppose that the regularity conditions in Assumptions~\ref{ass:Regu}, \ref{eq:Termi}, and \ref{ass:ContBoundary} hold. If Assumptions~\ref{ass:MonoTf}, \ref{ass:MonoSys}, and \ref{ass:MonoPath} also hold, then $\lambda^\top g^\star[t]> 0 $ and $u^\star(t) = \max  \mathcal{D}(x^\star(t))$ for a.e. $t \in [0,t_f]$, with $ \mathcal{D}$ defined in \eqref{eq:ControlSet}, and the optimal trajectory is bang-ride.
\end{thm}

Theorem~\ref{thm:mono} shows that the optimal input is the maximum feasible input under the monotonicity assumptions. An important special case is when there is only an input bound, as in the bang-bang control, then $u^\star=u_{\mathrm{max}}$. Assumptions~\ref{ass:MonoTf}, \ref{ass:MonoSys}, and \ref{ass:MonoPath} are analogous to the properties in \cite{taghavian2023selector,drummond2023constrained}. In this work, we take a different analytical approach by exploiting PMP. This allows us to address the regularity conditions formally and make an explicit connection to the bang-ride control. More importantly, it opens up new ways to generalize the result in Remark~\ref{rem:negativeLambda} and Section~\ref{sec:MonoGene}.

 \begin{rem}[Practical considerations] Assumptions~\ref{ass:MonoTf}, \ref{ass:MonoSys}, and \ref{ass:MonoPath} rely on the optimal trajectory and cannot be verified before knowing the optimal solution. One approach is to verify them a priori for any path constraint and any $(\hat{x},\hat{u})$ in a relevant region, determined by the application. A drawback is that not all practical constraints satisfy the conditions, e.g., $\bar{s}=u -u_{\mathrm{max}} \leqslant 0$ satisfies $\partial \bar{s}/\partial u>0$ but $\bar{s}=u_{\mathrm{min}} - u \leqslant 0$ has $\partial \bar{s}/\partial u<0$. A more practical approach is to verify these assumptions and Assumption~\ref{ass:ContBoundary} a posteriori, given a (possibly approximate) solution $(\hat{x},\hat{u})$, e.g., from the hybrid simulation. If the solution satisfies the assumptions and the necessary optimality condition $\hat{u}(t) = \max  \mathcal{D}(\hat{x}(t))$, then it is likely to be optimal.
 \end{rem}

  \begin{rem}[Generalization] \label{rem:negativeLambda} An analogous analysis can ensure $\lambda^\top g>0 $ via $\lambda<0$ and $g^\star<0$. The assumptions should be modified to consider $\partial \bar{z}^\star_i[t_f] / \partial x \geqslant 0$, $g^\star[t]<0$, $F_x^\star[t]$ is Metzler, $l^\star_x\geqslant 0$, $\partial \bar{s}_i^\star[t]/\partial u>0$, and $\partial \bar{s}_i^\star[t] / \partial x \geqslant 0$.  
 \end{rem}

  \begin{rem}[Linear systems]For a linear system \eqref{eq:LinearSys}, Assumption~\ref{ass:MonoSys} holds iff $B>0$ and $A$ is a Metzler matrix. Moreover, in \eqref{eq:OCdem}, the properties in the assumptions may be affected by a similarity transformation. For example, a linear system with $B<0$ can be transformed into an equivalent model with $B>0$. The challenge is to find a transformation to satisfy all assumptions simultaneously. Developing properties invariant to similarity transformation is desired, which is beyond the scope of this article.
 \end{rem}

 \subsection{Connection to Hybrid Simulation}
We make an explicit connection to the hybrid simulation algorithm.
\begin{prop} \label{prop:HybridSimu}
 Given an initial state $x(0)$ satisfying that $(x(0),u_{\mathrm{max}})$ does not activate any constraint except $u \in [u_{\mathrm{min}},u_{\mathrm{max}}]$, consider the solution trajectory $(\hat{x},\hat{u})$ of the hybrid simulation, and suppose that it satisfies Assumption~\ref{ass:Regu}. If every active constraint $\bar{s}_i$ satisfies $\partial \bar{s}_i(\hat{x}(t),\hat{u}(t)) / \partial u>0$ for a.e. $t$ within its active time interval, then $\hat{u}=\max  \mathcal{D}(\hat{x}(t))$ for a.e. $t \in [0,t_f]$.
\end{prop}
\begin{proof}
This result holds because an input $u>\hat{u}$ violates the constraint $\bar{s}_i$ due to $\partial \bar{s}_i(\hat{x}(t),\hat{u}(t)) / \partial u>0$.
\end{proof}

Proposition~\ref{prop:HybridSimu} shows that the solution of the hybrid simulation satisfies \eqref{eq:maxControl} for its state trajectory. Recall that, under the assumptions of Theorem~\ref{thm:mono}, \eqref{eq:maxControl} becomes a necessary optimality condition. Therefore, if these assumptions are satisfied by the solution, the solution also satisfies the necessary optimality condition \eqref{eq:maxControl}. In this case, the solution of the hybrid simulation is likely to be optimal.

Moreover, the hybrid simulation assumes that $(x(0),u_{\mathrm{max}})$ does not activate any constraint. Property~\eqref{eq:maxControl} suggests a way to generalize the algorithm: If $(x(0),u_{\mathrm{max}})$ activates constraints, the input can be selected according to \eqref{eq:maxControl}, i.e., taking the maximum feasible input.
 
 \section{Linear Systems with External Positivity for Costate Dynamics} \label{sec:MonoGene}
 \subsection{Introduction and General Results}
Some physical constraints do not satisfy the condition $\partial s/ \partial x \leqslant 0 $ in \eqref{eq:PathMono}. An effort is made in \cite{taghavian2023selector} to address this limitation; however, their result is limited to linear constraints. Moreover, instead of requiring $\lambda>0$, $\lambda^\top(t)g$ can remain positive even if some elements of $\lambda(t)$ become negative. We establish a novel connection between this phenomenon and the problem of external positivity:
\begin{prob}[External positivity] \label{prob:Costate}
For a time-varying dynamical system
\begin{align} \label{eq:external}
 \dot{\lambda}(t) = A(t) \lambda(t)+ B(t) p(t), \quad y(t)= g^\top (t) \lambda(t),
\end{align}
 under what conditions is $y(t) > 0$ for a.e. $t$?
\end{prob}

The system \eqref{eq:external} with a positive output $y$ is often said to be \textit{externally positive} \cite{drummond2023externally}. In this work, the dynamical system \eqref{eq:external} arises from the costate equation \eqref{PMP:costate}, and the switching function $\lambda^\top g$ in \eqref{eq:No Swiching} is the output. Problem~\ref{prob:Costate} is itself a complex and open research topic \cite{drummond2023externally,weller2023external}, where the results are often necessary conditions or sufficient conditions limited to special settings. Moreover, the existing results mainly focus on time-invariant systems. Therefore, we focus on simpler costate equations, arising from the optimal control problem of linear systems with $l\equiv 0$. The goal is to establish $\lambda^\top(t) g^\star[t]>0$ by restricting the costate dynamics.

Consider the optimal control \eqref{eq:OC} with the linear system instead:
\begin{equation} \label{eq:SimpleLinear}
\dot{x} = \mathrm{diag}(a_1,\dots,a_n) x + 1_{n\times 1} u.
\end{equation}
The only loss of generality here is that $B = 1_{n\times 1}$ does not contain zeros, as any linear system~\eqref{eq:LinearSys} can be equivalently transformed to have a diagonal $A$ and normalized entries in $B$. Note that \eqref{eq:SimpleLinear} already satisfies Assumption~\ref{ass:MonoSys} and is a monotone system.

With this linear system, the following quantity of the constraints in \eqref{eq:GeneralConstraints} is important:
\begin{defn}
The function $ p_{i,j}$, defined as
 \begin{equation}
 p_{i,j}(x(t),u(t)): =    \left[\frac{\partial \bar{s}_i(x(t),u(t))}{ \partial u}\right]^{-1}  \frac{\partial \bar{s}_i(x(t),u(t))}{ \partial {x_j}}, 
\end{equation}
is the \textit{relative sensitivity} of the constraint $\bar{s}_i$ with respect to state $x_j$, where $i \in \{1,2,\dots,c_s+c_h\}$ and $j \in \{1,\dots,n\}$.
\end{defn}

The sign of $p_{i,j}$ indicates whether the input and the state affect the constraint in the same direction. Its magnitude shows whether the constraint is more sensitive to input or state.

Then we replace Assumptions~\ref{ass:MonoTf} and \ref{ass:MonoPath} by the following assumption:
\begin{ass} \label{ass:Linear}
For the optimal control problem \eqref{eq:OC} with the system~\eqref{eq:SimpleLinear} and its optimal trajectory:
\begin{enumerate}[label=(\alph*)]
 \item There exist $j_1$, $j_2 \in \{1,\dots,n\}$ such that\footnote{Note that $\lambda(t_f^-) =\lambda(t_f)$ under Assumption~\ref{ass:ContBoundary}(a), as $\lambda(t)$ becomes left continuous. Recall that when computing $\lambda(t^-_f)$ from \eqref{PMP:costate-boundary}, $\lambda_0$ is often normalized to be $-1$ if non-zero.} $\lambda_{j_1}(t_f^-) > \lambda_{j_2}(t_f^-)$ and $a_{j_1} >a_{j_2}$.

\item $a_k \geqslant a_j$ if $\lambda_k(t_f^-)  \geqslant \lambda_j(t_f^-)$, for any $k,j \leqslant n$.

    \item For every $\bar{s}_i$ active for a non-zero time interval, the conditions
    \begin{enumerate}[label=(\roman*)]   
        \item $ 
\partial \bar{s}_i^\star[t]/ \partial u > 0$,
\item $p^\star_{i,k}[t]\leqslant p^\star_{i,j}[t] $ if $\lambda_k(t_f^-)  \geqslant \lambda_j(t_f^-)$, for any $k,j \leqslant n$,
    \end{enumerate}
hold for a.e. $t$ within its active time interval.
\end{enumerate}
\end{ass}

With the above assumption, the result in Theorem~\ref{thm:mono} remains true.

\begin{thm}  \label{thm:linear}
For the optimal control problem \eqref{eq:OC} with the system~\eqref{eq:SimpleLinear} and its optimal trajectory $(x^\star,u^\star)$, suppose that Assumptions~\ref{ass:Regu}, \ref{eq:Termi}a, and \ref{ass:ContBoundary} hold. If Assumption~\ref{ass:Linear} also holds, then $\lambda^\top g^\star[t]> 0 $ and $u^\star(t) = \max  \mathcal{D}(x^\star(t))$ for a.e. $t \in [0,t_f]$, with $ \mathcal{D}$ defined in \eqref{eq:ControlSet}, and the optimal trajectory is bang-ride.
\end{thm}

Assumption~\ref{ass:Linear} is designed so that the combined effect of the terminal condition $\lambda(t_f)$, the system parameters $a_i$, and the constraints leads to an externally positive costate dynamics. Moreover, the costate dynamics is order-preserving, i.e., $\lambda_k(t_f) \geqslant \lambda_j(t_f)$ implies $\lambda_k(t) \geqslant \lambda_j(t)$ for all $t\leqslant t_f$, as shown in the proof of Theorem~\ref{thm:linear}.

Here, we compare Theorems~\ref{thm:mono} and~\ref{thm:linear}. First, the system class \eqref{eq:SimpleLinear}  in Theorem~\ref{thm:linear} satisfies Assumption~\ref{ass:MonoSys} in Theorem~\ref{thm:mono}. In Theorem~\ref{thm:linear}, $\lambda(t^-_f)$ may contain both positive and negative elements, unlike $\lambda(t_f^-) \geqslant 0$ in Theorem~\ref{thm:mono} due to Assumption~\ref{ass:MonoTf}. Moreover, Theorem~\ref{thm:linear} does not restrict the sign of $\partial \bar{s}_i^\star/ \partial x$. Also, note that compared to Theorem~\ref{thm:Main}, controllability is not required in Theorem~\ref{thm:linear} as Assumption~\ref{ass:Linear} allows identical diagonal elements in $A$.

The above generality of Theorem~\ref{thm:linear} is achieved by the additional restrictions on the system in Assumption~\ref{ass:Linear}b and on the relative sensitivity of the constraints in Assumption \ref{ass:Linear}c. Nonetheless, we view Theorem~\ref{thm:linear} as an important foundation for a more general solution in future research.

\begin{exam} \label{exam:Analysis}
Recall problem~\eqref{eq:SpecialOpti2}, whose system is in the system class~\eqref{eq:SimpleLinear}, and the solutions of the hybrid simulation in Fig.~\ref{fig:Special} for $a_1=0$, $a_2=-1$ and in Fig.~\ref{fig:Special2} for $a_1=-1$, $a_2=0$. From its terminal objective $\phi(t_f)=-x_1(t_f)$ and \eqref{PMP:costate-boundary}, we have $\lambda(t_f) = [1 \text{ } 0]^\top$, showing $\lambda_1(t_f)=1 > \lambda_2(t_f)=0$. Then, the case with  $a_1=0$ and $a_2=-1$ satisfies Assumptions~\ref{ass:Linear}ab due to $a_1>a_2$. However, the case with $a_1=-1$ and $a_2=0$ does not satisfy these assumptions.

The solutions of the hybrid simulation in both cases satisfy Assumption~\ref{ass:Linear}c: First, as $x(0)>0$ and $u\geqslant 0$, the system structure shows $x_1$ and $x_2$ are positive for all $t$ under the hybrid simulation. Therefore, $\partial \bar{s} /\partial u = x_1+x_2 >0$ holds for a.e. $t$ and thus satisfies Assumption~\ref{ass:Linear}c(i), where $\bar{s}$ denotes the constraint~\eqref{eq:SpecialOptiCons}. Second, the relative sensitivity of constraint~\eqref{eq:SpecialOptiCons} satisfies $p_{i,1} =p_{i,2} = (x_1+x_2)^{-1} u$ and Assumption~\ref{ass:Linear}c(ii).   

In summary, if $a_1=0$ and $a_2=-1$, the hybrid simulation satisfies Assumption~\ref{ass:Linear} and consequently also the necessary optimality condition based on Proposition~\ref{prop:HybridSimu}. Therefore, the hybrid simulation likely provides the optimal solution, which indeed matches the optimal solution in Fig.~\ref{fig:Special}. However, if $a_1=-1$ and $a_2=0$, the hybrid simulation does not satisfy Assumptions~\ref{ass:Linear}ab. The hybrid simulation is not optimal in that case, as shown in Fig.~\ref{fig:Special2}. \hfill$\triangle$
\end{exam}

 \subsection{Special Setting for Application} \label{sec:specialLinear}
We also consider a more special setting, where the terminal objective and terminal constraint are functions of a single state. Despite appearing restrictive, these settings are important in real-world applications, as shown in Section~\ref{sec:applyAnalysis}.

\begin{ass} \label{ass:SpecialLambda} 
For the optimal control problem \eqref{eq:OC} with the system~\eqref{eq:SimpleLinear}:
\begin{enumerate}[label=(\alph*)]
    \item $\lambda_1(t_f^-)>0$ and $\lambda_j(t_f^-) =0$ for all $j \in \{2,\dots,n\}$.

    \item $a_1 \geqslant a_j$, for all $j \in \{2,\dots,n\}$. Moreover, $a_1>a_k$ for some $k \in \{2,\dots,n\}$, where $x_k$ is the last (in time) state satisfying\footnote{This property is satisfied trivially if no such state exists.} $ \bar{s}^\star_j[t]/\partial x_k >0$ for a non-zero time interval and for some active constraint $\bar{s}^\star_j$.
 
    \item For every $\bar{s}_i$ that is active, the inequalities
 \begin{equation}  \label{eq:PosiS}
\frac{\partial \bar{s}_i^\star[t]}{ \partial u} > 0, \text{ } \frac{\partial \bar{s}_i^\star[t] }{ \partial x} \geqslant 0,    \end{equation}
hold for a.e. $t$ in its active time interval.
\end{enumerate}
\end{ass}

Assumption~\ref{ass:SpecialLambda}a shows the terminal ingredients only depend on $x_1$, based on \eqref{PMP:costate-boundary}. Condition~\ref{ass:SpecialLambda}b further restricts the corresponding system dynamics. Specifically, it requires a strict inequality $a_1 >a_k$ for the last state $x_k$ that has a positive effect on an active constraint. Finally, Assumption~\ref{ass:SpecialLambda}c allows $\partial \bar{s}_i^\star/ \partial x$ to be positive, in contrast to Assumption~\ref{ass:MonoPath}. 

With the above properties, the following result shows that the optimal input remains the maximum feasible input.
\begin{prop}  \label{prop:SpecialLinear}
For the optimal control problem \eqref{eq:OC} with the system~\eqref{eq:SimpleLinear} and its optimal trajectory $(x^\star,u^\star)$, suppose that Assumptions~\ref{ass:Regu}, \ref{eq:Termi}a, and \ref{ass:ContBoundary} hold. If Assumptions~\ref{ass:SpecialLambda} also holds, then $\lambda^\top g^\star[t]> 0 $ and $u^\star(t) = \max  \mathcal{D}(x^\star(t))$ for a.e. $t \in [0,t_f]$, with $ \mathcal{D}$ defined in \eqref{eq:ControlSet}, and the optimal trajectory is bang-ride.
\end{prop}

Analogous to Example~\ref{exam:Analysis}, it can be shown that problem~\eqref{eq:SpecialOpti2} satisfies Assumption~\ref{ass:SpecialLambda}  and thus Proposition~\ref{prop:SpecialLinear} when $a_1=0$ and $a_2=-1$; however, the case with $a_1=-1$ and $a_2=0$ does not satisfy Assumption~\ref{ass:SpecialLambda}b.

\section{Application Example: Fast Battery Charging} \label{sec:Apply}
\subsection{Charging with the Single-Particle Model (SPM)}
The theory is illustrated by applying it to a battery fast-charging problem using the SPM \cite{liu2016extended}. The SPM describes the evolution of lithium concentration $c_j(r,t)$ over a spatial coordinate $r \in [0,R_j]$:
\begin{align} \label{eq:diffusion}
\frac{\partial c_j(r,t)}{\partial t} = \frac{D_{j,s}}{r^2} \frac{\partial}{ \partial r}\!\left( r^2 \frac{\partial c_{j}(r,t)}{\partial r } \right), 
\end{align}
where $D_{j,s}$ is a positive parameter and $j \in \{+,-\}$ denotes the index of the positive or negative electrode. The boundary conditions are
$$\frac{\partial c_\pm(0,t) }{\partial r} =0, \quad \frac{\partial c_\pm(R_\pm,t) }{\partial r} = \mp \frac{   I(t) }{m_\pm}, $$
where $m_j$ is a positive parameter, and $I$ is the input current to be controlled. The sign convention here is $I \geqslant 0$ for charging.

For our purposes, only the surface concentration $c_{j,\mathrm{s}}(t):=c_j(R_j,t)$ and the average concentration $c_{j,\mathrm{ave}}$ \cite{forman2010reduction} are of interest. Regarding these two quantities as two outputs leads to an approximation of \eqref{eq:diffusion} in 
the form of a linear state-space model\footnote{While each electrode is approximated by a $3$-dimensional linear system, leading to $6$ states in total, both electrodes contain an integrator, and thus one redundant state for integration is eliminated.}  \cite{park2020optimal,forman2010reduction}:
\begin{subequations} \label{eq:SPMmodelCombine}
\begin{align}  
\dot{x} &= \begin{bmatrix}
     A_+ & 0 \\
     0 & A_- 
 \end{bmatrix} x+ \begin{bmatrix}
     B_+\\
     B_-
 \end{bmatrix} I ,\\ \label{eq:SPMmodelCombineB}
\begin{bmatrix}
 c_{+,\mathrm{ave}} \\
 c_{+,\mathrm{s}} \\
 c_{-,\mathrm{ave}} \\
 c_{-,\mathrm{s}} \\
 \end{bmatrix} &=  \begin{bmatrix}
C_{+,\mathrm{ave}} \\
C_{+,\mathrm{s}} \\
 C_{-,\mathrm{ave}}\\
 C_{-,\mathrm{s}} \\
 \end{bmatrix} x,
\end{align}
\end{subequations}
where $x(t) \in \mathbb{R}^5$ and
\begin{align}
A_+ &= \mathrm{diag}(0,-0.0514,-0.4211), \nonumber \\
C_{+,s} &= \begin{bmatrix}
 -0.1639&    -0.1193& -0.8643 & 0_{1\times 2}
\end{bmatrix},  \nonumber \\
C_{+,\mathrm{ave}} &= \begin{bmatrix}
 -0.1639&    0_{1\times 4}
\end{bmatrix}, \label{eq:PosSS-Caver} \\
B_+&= 1_{3\times 1}, \text{ }
B_- = 1_{2\times 1},  \nonumber\\
A_- &= \mathrm{diag}(-0.2006,-1.6422), \nonumber\\
C_{-,s} &= \begin{bmatrix}
 0.1183& 0_{1\times 2} &0.0861 & 0.6237
\end{bmatrix},\nonumber \\
C_{-,\mathrm{ave}} &= \begin{bmatrix}
0.1183&  0_{1\times 4}
\end{bmatrix}.\nonumber
\end{align}
Note that $x_1$ is the integration of the input, due to the zero diagonal entry in $A_+$. 
The above parameters are computed following the model reduction procedure in \cite{forman2010reduction} and the physical parameter values in \cite{torchio2016lionsimba}. A similarity transformation is also conducted to obtain the model in the form of \eqref{eq:SimpleLinear}.

The battery voltage $V$ is a nonlinear function \cite{moura2016battery}:\footnote{This is consistent with the voltage function in \cite{moura2016battery} with its last two terms for the electrolyte dynamics removed.} 
\begin{align}
V(c_{+,\mathrm{s}},c_{-,\mathrm{s}},I)  = & \text{ }\eta_{+ }(c_{+,\mathrm{s}}, I) - \eta_{- }(c_{-,\mathrm{s}} \nonumber, I)\\& + U_+\left(c_{+,s}(t)\right)- U_-\left(c_{+,s}(t) \right)+R_f I(t) \nonumber,
\end{align}
where $\eta_j$ denotes the overpotential, $U_j$ is the open-circuit potential \cite{torchio2016lionsimba}, and $R_f$ is a positive parameter. These functions and their parameters are presented in Appendix~G. Another important quantity is the state-of-charge (SOC):
$$
\mathrm{SOC}(c_{+,\mathrm{ave}}) = \frac{c_{+,\mathrm{ave}}/51554 - 0.9917}{ -0.4962},
$$
where $\mathrm{SOC} \in [0,1]$. Given an initial $\mathrm{SOC}$, an initial state can be computed following \cite{forman2010reduction} for simulation.

Consider the voltage-constrained fast-charging problem:
\begin{align} \label{eq:SPMcontrol}
 \min_{I} \text{ } &  - \mathrm{SOC}(c_{+,\mathrm{ave}}(t_f) )   \\
 \text{s.t. } & 
\eqref{eq:SPMmodelCombine} 
 \nonumber \\
 & I(t) \in [0, I_{\mathrm{max}}] \\
 & V(c_{+,\mathrm{s}}(t),c_{-,\mathrm{s}}(t),I)  \leqslant V_{\mathrm{max}}, 
\end{align}
with $I_{\mathrm{max}}=300 \text{ }\mathrm{A}$, $V_{\mathrm{max}} = 4.5 \text{ }\mathrm{V}$, and $t_f$ fixed. The goal is to increase SOC as fast as possible while limiting the current and voltage. Another formulation is a time-optimal problem with $t_f$ also being an optimization variable. The analysis of this formulation is analogous, but optimizing $t_f$ causes additional numerical challenges.

\subsection{Analysis of the Fast-Charging Problem} \label{sec:applyAnalysis}
It is well-known from numerical results \cite{berliner2022novel} that the optimal input is the CC-CV profile, which is consistent with the hybrid simulation, as also shown in Fig.~\ref{fig:SPM}. We provide a novel theoretical justification for this using Proposition~\ref{prop:SpecialLinear}.
 \begin{figure}[h]
    \centering
\hspace*{-0.6cm} \includegraphics[width=0.55\textwidth]{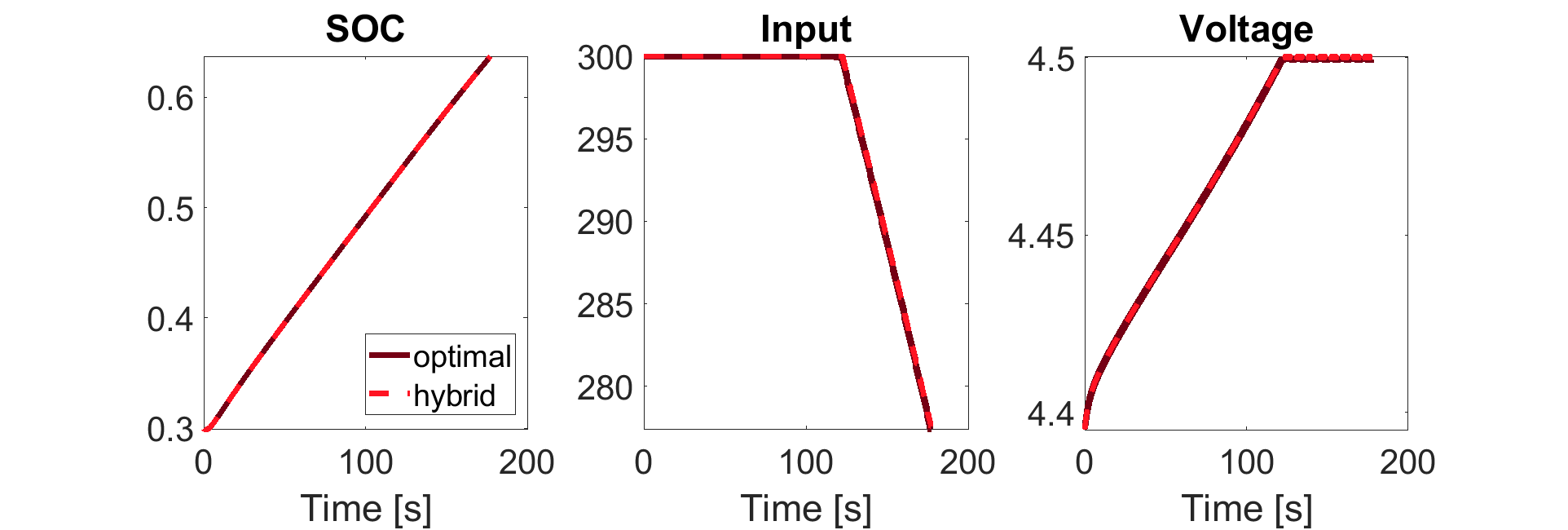}
    \caption{The optimal solution, the practical CC-CV profile, and the hybrid simulation are consistent.}
    \label{fig:SPM}
\end{figure}
 
 Recall that model~\eqref{eq:SPMmodelCombine} is in the form of \eqref{eq:SimpleLinear} and is a monotone system. For Assumption~\ref{ass:SpecialLambda}a, $\mathrm{SOC}$ depends on $c_{+,\mathrm{ave}}$ and thus only on the state $x_1$ due to \eqref{eq:PosSS-Caver}. We can compute $\lambda(t_f^-)$ from \eqref{PMP:costate-boundary} and thus from $\partial \mathrm{SOC}/ \partial x$. First, SOC satisfies
$$\frac{\partial \mathrm{SOC}}{\partial c_{+,\mathrm{ave}}}< 0,$$
i.e., the lithium concentration of the positive electrode decreases during charging. Thus, we have
\begin{equation} \label{eq:SOCpartial}
\frac{\partial \mathrm{SOC}}{\partial x} =\frac{\partial \mathrm{SOC}}{\partial c_{+,\mathrm{ave}}} \frac{\partial c_{+,\mathrm{ave}}}{ \partial x} = \frac{\partial \mathrm{SOC}}{\partial c_{+,\mathrm{ave}}} \begin{bmatrix}
  -0.1639 & 0_{1\times 4}  
\end{bmatrix} \geqslant 0.
\end{equation}
Combining the above with $\lambda_0=-1$ and \eqref{PMP:costate-boundary} shows $\lambda_1(t_f^-)>0$ and $\lambda_i(t_f^-)=0$ for all $i>1$, satisfying Assumption~\ref{ass:SpecialLambda}a.

For Assumption~\ref{ass:SpecialLambda}b, the system parameters satisfy
$a_1=0 \geqslant a_j$ for any $j \geqslant 2$. Moreover, for the solution of the hybrid simulation and the CC-CV profile, the voltage function is the last active constraint and satisfies the physical properties:
 \begin{equation} \label{eq:Vproperty}
  \frac{\partial V}{\partial c_{+,\mathrm{s}}}< 0,\quad \frac{\partial V}{\partial c_{-,\mathrm{s}}}> 0,\quad\frac{\partial V}{\partial I }> 0.
 \end{equation}
Analogous to \eqref{eq:SOCpartial}, combining \eqref{eq:Vproperty} and \eqref{eq:SPMmodelCombineB} shows
\begin{equation} \label{eq:Vproperty2}
\frac{\partial V}{ \partial x} > 0 ,\quad \frac{\partial V}{\partial I }>0.
\end{equation}
Then there exists $k=2$ such that $a_1=0>a_2 =-0.0514$, where $V$ is the last active constraint and $x_2$ satisfies $\partial V / \partial x_2>0$ when $V$ is active. Therefore, the CC-CV profile and hybrid simulation satisfy Assumption~\ref{ass:SpecialLambda}b. Note that choosing any $k\in\{2,3,4,5\}$ satisfies the above condition.

For Assumption~\ref{ass:SpecialLambda}c, the hybrid simulation and CC-CV profile activate $I = I_{\mathrm{max}}$ and the voltage constraint, satisfying Assumption~\ref{ass:SpecialLambda}c based on \eqref{eq:Vproperty2}.

In summary, Assumption~\ref{ass:SpecialLambda} is satisfied, and thus $u=\max\mathcal{D}(x)$ becomes a necessary optimality condition, based on Proposition~\ref{prop:SpecialLinear}. Then, the hybrid simulation and the CC-CV profile satisfy this necessary optimality condition, based on Proposition~\ref {prop:HybridSimu}. Although its optimality is well-known empirically, the above formal theoretical analysis is novel.

The bang-ride behavior can also be explained by Theorem~\ref{thm:Main} as \eqref{eq:SPMmodelCombine} is controllable. However, Theorem~\ref{thm:mono} and the result in \cite{taghavian2023selector} are not applicable, as they require the nonlinear constraint to be decreasing in state. The analysis in \cite{park2020optimal} is limited to linear constraints and thus is not applicable here.

\section{CONCLUSIONS}
This work has provided theoretical conditions under which heuristic input profiles and hybrid simulation methods in time-critical control problems coincide with necessary optimality. Using Pontryagin’s maximum principle, we showed that bang-ride structures arise under controllability and regular constraint switching, and that maximum-feasible inputs become necessary under monotonicity or external positivity of the costate dynamics. These results justify widely used charging heuristics and provide efficient, optimization-free certificates for their optimality.

Future research should address settings with mixed monotonicity, where some constraints increase and others decrease with the state, as well as system classes beyond finite-dimensional linear dynamics. Such generalizations may require tailored notions of external positivity or problem-specific sufficient conditions, and could extend the applicability of these results to PDE systems or nonlinear dynamics with input nonlinearities.

\bibliographystyle{IEEEtran} 
\bibliography{hySi.bib}

\section*{Appendix}
\subsection{Discussion: Nonlinear Controllability, Feedback Linearization, and Differential Flatness} \label{Apx:SystemProper}
The system properties in Assumption~\ref{ass:Control} can be interpreted in multiple ways. The properties reduce to controllability when the system is linear, as $M_n$ becomes the controllability matrix. For nonlinear systems \eqref{eq:sys}, there are various forms of ``controllability". It has been noted in \cite{palanki1993synthesis} that the system considered in Assumption~\ref{ass:Control} is locally controllable in the sense of \cite{hunt1982sufficient}. Instead, we will make a connection to local accessibility by following \cite{nijmeijer1990nonlinear}. 

 Let $R^V(x_0,T)$ be the reachable set from state $x_0$ at time $T>0$ following trajectories which remain in a neighborhood $V$ of $x_0$ for $t\leqslant T$. 
\begin{defn}[\cite{nijmeijer1990nonlinear}]
 System \eqref{eq:sys} is said to be \textit{locally strongly accessible} at $x$ if for any $V$, $R^V(x,T)$ contains a non-empty open set for any $T>0$ sufficiently small.
\end{defn}

This property means that, from every state, the system can reach an open region close to the starting point, indicating a capability of being controlled. Note that this is a local concept. 

Locally strong accessibility can be determined by the rank (dimension) test of a nonlinear generalization of the controllability matrix (controllable space), called \textit{strong accessibility algebra}, whose formal definition is in \cite{nijmeijer1990nonlinear}. Then Assumption~\ref{ass:Control} implies locally strong accessibility.

\begin{prop} \label{prop:Controllability}
If Assumption \ref{ass:Cinf} is satisfied, the following holds:
\begin{enumerate}[label=(\alph*)]
    \item if Assumption~\ref{ass:Control}a holds, $\mathcal{L} =  \mathrm{span}\big\{g,\text{ }\mathrm{ad}_f g, \text{ }\mathrm{ad}^2_f g,\dots,\text{ }\mathrm{ad}^v_f g \mid v = 1,2,\cdots \big \}$ is the strong accessibility algebra of system \eqref{eq:sys}.

 \item if Assumption~\ref{ass:Control} holds, system \eqref{eq:sys} is locally strongly accessible at any $x\in \left\{x^\star(t)\mid t \in [0,t_f] \right\}$. 
\end{enumerate}
\end{prop}
\begin{proof}
  Based on \cite[Prop. 3.20]{nijmeijer1990nonlinear}, every element of the strong accessibility algebra has the form 
 $$
 [X_k,[X_{k-1},[\cdots,[X_1,g]\cdots ]]], \ \ k = 0, 1, \dots,
 $$
 where $X_i \in \{f,g\}$. For $k=0$, the element is $g$ and, for $k=1$, the algebra contains $[g,g]=0$ and $[f,g] = \mathrm{ad}_f g$. Similarly, the algebra contains the elements $[g, \mathrm{ad}_f g]$ and $\mathrm{ad}^2_f g$ for $k=2$. Due to condition~(a), $\mathrm{span}\{g,\mathrm{ad}_f g, [g, \mathrm{ad}_f g],  \mathrm{ad}^2_f g\} = \mathrm{span}\{g,\mathrm{ad}_f g,  \mathrm{ad}^2_f g\}$. Following an inductive analysis, the first statement is proved. Then condition (b) implies that $\mathcal{L}(x)$ has dimension $n$ for any $x \in \left\{x^\star(t)\mid t \in [0,t_f] \right\}$, proving locally strong accessibility.    
\end{proof}

The properties in Assumption~\ref{ass:Control} are stronger than locally strong accessibility; also required is a special structure of the strong accessibility algebra via Assumption~\ref{ass:Control}a.

As noted in \cite{benthack1997feedback}, systems satisfying Assumption~\ref{ass:Control} are also connected to the feedback linearizable system \cite{nijmeijer1990nonlinear}. We emphasize here that Assumption~\ref{ass:Control} actually points to a subset of feedback linearizable systems. 
\begin{coro}
 Consider a single-input system \eqref{eq:sys} with $f(\bar{x})=0$ and Assumption~\ref{ass:Cinf} satisfied. If the system satisfies Assumption~\ref{ass:Control}, with $M_n$ having rank $n$ instead, for any $x$ in a neighborhood of $\bar{x}$, then the system is feedback linearizable around $\bar{x}$.   
\end{coro}
\begin{proof}
 Directly follows from \cite[Cor. 6.17]{nijmeijer1990nonlinear} and Assumption~\ref{ass:Control}. 
\end{proof}
 
This result further connects to a subset of \textit{differentially flat systems} \cite{beaver2024optimal}, as this property is equivalent to feedback linearization\footnote{Note that for single-input systems, dynamic feedback linearization and static feedback linearization are equivalent \cite{charlet1989dynamic}.} \cite{levine2007equivalence}.

\subsection{Proof of Proposition~\ref{prop:NonZeroLambda}}
 Assumption~\ref{ass:ContBoundary}b leads to continuous $\lambda$ in $(0,t_f)$, e.g., see \cite[Thm.~4.2]{hartl1995survey} and \cite[Eq.~(71) in Chap. 6]{seierstad1986optimal}. Moreover, Assumption~\ref{ass:ContBoundary}a shows that $\lambda$ is right continuous at $0$ and left continuous at $t_f$. Therefore, $\lambda$ is continuous at any $t \in [0,t_f]$, leading to $\eta(\tau_i) = 0$ for any $i$. Besides, recall that the effect of $h^\star_x$ on $\lambda(t_f^-)$ in \eqref{PMP:costate-boundary} is caused by the switching of state constraints at $t_f$, following the analysis in \cite[Appendix~B]{bryson1963optimal}. Then Assumption~\ref{ass:ContBoundary} also implies $\gamma =0$ in \eqref{PMP:costate-boundary}. 

Equation \eqref{PMP:Lderive} shows that 
$$
\lambda^\top F_u^\star + \begin{bmatrix}
    \mu^\top & \nu^\top
  \end{bmatrix} \begin{bmatrix}
     s^\star_u \\
     \dot{h}^\star_u
  \end{bmatrix}  
   =0. 
$$
As the multiplier corresponding to an inactive constraint is zero, it also holds that
\begin{equation} \label{eq:prop31-pf1}
\lambda^\top \begin{bmatrix}
    F_u^\star & 0 & 0
    \end{bmatrix}+ \begin{bmatrix}
    \mu^\top & \nu^\top
  \end{bmatrix} \underbrace{\begin{bmatrix}
     s^\star_u & \mathrm{diag}(s^\star) & 0 \\
     \dot{h}^\star_u & 0 & \mathrm{diag}(h^\star)
  \end{bmatrix}}_{P}  
   =0. 
\end{equation}
Further define the matrix 
$$
G := P^\top (P P^\top)^{-1},
$$
where the inverse exists as $P$ has full row rank based on Assumption~\ref{ass:Regu}. Then post-multiplying both sides of \eqref{eq:prop31-pf1} by $P^\top$ leads to 
\begin{align} 
\begin{bmatrix}
    \mu^\top & \nu^\top
  \end{bmatrix} & =   -   \lambda^\top \begin{bmatrix}
    F_u^\star & 0 & 0
    \end{bmatrix} G \nonumber \\
& =  - \lambda^\top  F_u^\star  \label{eq:prop31-pf2}\underbrace{\begin{bmatrix}
  (s_u^\star)^\top & (\dot{h}^\star_u)^\top  
\end{bmatrix} (PP^\top)^{-1}}_{\bar{G}}.
\end{align}

Then \eqref{PMP:costate} shows that 
\begin{equation*}
 \dot{\lambda}^\top = -\lambda_0 l_x^\star - \lambda^\top F_x^\star -  \begin{bmatrix}
    \mu^\top & \nu^\top
  \end{bmatrix} \begin{bmatrix}
     s^\star_x \\
     \dot{h}^\star_x
  \end{bmatrix},  
\end{equation*}
which together with \eqref{eq:prop31-pf2} shows that
\begin{equation} \label{eq:prop31-pf3}
  \dot{\lambda}^\top  = -\lambda_0 l_x^\star - \lambda^\top \!\left( F_x^\star - F_u^\star \bar{G} \begin{bmatrix}
     s^\star_x \\
     \dot{h}^\star_x
  \end{bmatrix} \right).  
\end{equation}

If Assumption~\ref{eq:Termi}a holds with $l \equiv 0$, then \begin{equation}\label{eq:prop31-pf4}
\dot{\lambda} = - \lambda^\top \!\left( F_x^\star - F_u^\star \bar{G} \begin{bmatrix}
     s^\star_x \\
     \dot{h}^\star_x
  \end{bmatrix} \right).
\end{equation}
Assuming the existence of $t_k$ such that $\lambda(t_k)=0$, \eqref{eq:prop31-pf4} and the continuity of $\lambda$ show $\lambda(t) =0$ for all $t\in [t_k,t_f]$, and thus $0 = \lambda_0 \phi^\star_x[t_f] + \alpha^\top z_x^\star[t_f]$ holds from \eqref{PMP:costate-boundary}. Then \eqref{eq:Arank2} shows $\lambda_0 = 0$ and $\alpha=0$, leading to $(\lambda_0,\lambda(t_f), \alpha,\gamma, \eta(\tau_1),\dots) =0$ and contradicting the nontriviality condition \eqref{eq:NonTrivial}.

If Assumption~\ref{eq:Termi}b holds, $\lambda_0 l^\star+ \lambda^\top F^\star= 0$ from \eqref{PMP:ConstantH}. Suppose $\lambda(t_k)=0$ for some $t_k$, $H=0$ and $l^\star \not=0$ shows $\lambda_0=0$. Thus $\lambda$ satisfies \eqref{eq:prop31-pf4}, leading to $\lambda(t) = 0$ for all $t \in [t_k, t_f]$, given the continuity of $\lambda$. Then \eqref{PMP:costate-boundary} leads to $0 = \alpha^\top z_x^\star[t_f]$ due to $\gamma =0$, which shows $\alpha = 0$ due to \eqref{eq:Arank}. Therefore, we have obtained $(\lambda_0,\lambda(t_f), \alpha,\gamma, \eta(\tau_1),\dots) =0$, contradicting the nontriviality condition. This shows that $\lambda(t) \not=0$ for any $t$.

If Assumption~\ref{eq:Termi}c holds and if $\lambda(t)=0$ for some non-zero interval, then $0=\dot{\lambda}^\top = -\lambda_0 l_x^\star - \lambda^\top F_x^\star=-\lambda_0 l_x^\star $ and $l_x^\star[t] \not=0$ show $\lambda_0=0$. Therefore,  $\lambda(t_f) = \alpha^\top z_x^\star[t_f] = 0$, which leads to $(\lambda_0,\lambda(t_f), \alpha,\gamma, \eta(\tau_1),\dots) =0$, contradicting the nontriviality condition. This shows $\lambda\not=0$ for any non-zero time interval and thus for a.e. $t$.

\subsection{Proof of Theorem~\ref{thm:mono}}
Due to the scalar input, Assumption~\ref{ass:Regu}b shows that at most one constraint $\bar{s}_{i(t)}$ is active at a time instant, where $i(t)$ denotes the index of the active constraint at $t$. Following \eqref{eq:prop31-pf3}, we have that 
\begin{equation} \label{eq:mui}
    \mu_{i(t)}= -\lambda^\top g^\star [\bar{s}_{i(t)}]_u^{-1},
\end{equation}
where $\mu_{i(t)} = \nu_{i(t)-c_s}$ if $i(t) \in \{ c_s+1,\dots,c_s+c_h\}$.
Note that this term is zero if no constraint is active at $t$. Then \eqref{eq:mui} shows that $\lambda$ satisfies 
\begin{equation}  \label{eq:proof-lambda}
  \dot{\lambda}^\top  = -\lambda_0 l_x^\star[t] - \lambda^\top \underbrace{\left( F^\star_x[t] - g^\star[t]  [\bar{s}_{i(t)}]_u^{-1} [\bar{s}_{i(t)}]_x  \right)}_{\Gamma^\star[t]}.  
\end{equation}
Moreover, $\lambda$ is continuous and $\gamma=0$ in \eqref{PMP:costate-boundary} given Assumption~\ref{ass:ContBoundary}, following the proof of Proposition~\ref{prop:NonZeroLambda}. Then \eqref{PMP:costate-boundary-full} and Assumption~\ref{ass:MonoTf} show 
$$
\lambda^\top(t_f^-) =\lambda^\top(t_f)= \lambda_0 \phi^\star_x[t_f]+ \alpha^\top z^\star_x[t_f] \geqslant 0.
$$
The above terminal condition and \eqref{eq:proof-lambda} can be transformed into an initial value problem by defining $	\zeta(s):=\lambda(t_f-s)$ with $s\in [0,t_f ]$, leading to 
$$
\dot{	\zeta}^\top(s) = \lambda_0l_x^\star[t_f-s]+ 	\zeta^\top(s)\Gamma^\star[t_f-s], \text{ }	\zeta(0) =\lambda(t_f) \geqslant 0.
$$
Then Assumptions~\ref{ass:MonoSys} and \ref{ass:MonoPath} show that $\Gamma^\star[t]$ is Metzler and $\lambda_0 l_x^\star[t]\geqslant 0$ for a.e. $t\in [0,t_f]$. Then, based on \cite[Lemma VIII.1]{angeli2003monotone}, the differential equation of $\zeta$ is a positive system, i.e., $\zeta(0) \geqslant 0$ leads to $\zeta(s) \geqslant 0$ for a.e. $s\in [0,t_f]$, and thus $\lambda(t) \geqslant 0$ for a.e. $t\in [0,t_f]$. Also note that $\lambda(t) \not=0$ for a.e. $t$ under Assumption~\ref{eq:Termi}, according to Proposition~\ref{prop:NonZeroLambda}. Combining the above with the assumption $g^\star[t]>0$ proves the result. 

Finally, note that the analysis for \eqref{eq:proof-lambda} also holds for $[\bar{s}_{i(t)}]_u^{-1} \leqslant 0$ and $\bar{s}_x \geqslant 0$. However, as $\mu_{i(t)} \leqslant 0$ and $\lambda^\top g>0$, \eqref{eq:mui} shows that $[\bar{s}_{i(t)}]_u^{-1} \leqslant 0$ cannot be considered.

\subsection{Proof of Corollary~\ref{coro:ControlMax}}
Eq. \eqref{eq:mui} shows $ \lambda^\top g \geqslant 0$ due to $[\bar{s}_{i(t)}]_u^{-1}>0$ and $\mu_{i(t)} \leqslant 0$. Then the result follows from $\lambda^\top g \not=0$ in Theorem~\ref{thm:Main}.

\subsection{Proof of Theorem~\ref{thm:linear} }
Starting from \eqref{eq:proof-lambda}, due to $l \equiv 0$ in Assumption~\ref{eq:Termi}a, we have 
\begin{equation} \label{eq:proofCostate}
 \dot{\lambda}= -  (A - p [t] 1^\top ) \lambda,   
\end{equation}
where $p^\star [t] := [\bar{s}^\star_{i(t)}]_u^{-1} [\bar{s}^\star_{i(t)}]_x^\top$ is a column vector and $1$ denotes $1_{n\times 1}$ for brevity, and $A$ is diagonal based on \eqref{eq:SimpleLinear}. Define 
$$	\psi(t): = 1^\top \lambda(t),$$
and \eqref{eq:proofCostate} shows
\begin{equation}  \label{eq:proofCostate2}
\dot{\lambda}_k = -a_k\lambda_k+ p^\star_{k}[t]	\psi(t),
 \end{equation}
 where $p^\star_{k}[t] = p^\star_{i(t),k}(t)$ is the relative sensitivity. We aim to prove $1^\top\lambda(t)>0$ for a.e. $ t \in [0,t_f]$. Note that $\psi(t)=1^\top\lambda(t) \geqslant 0$ must hold: If $1^\top\lambda(t) \not=0$, a constraint is active with $\partial \bar{s}^\star[t]_i/\partial u>0$ based on Assumption~\ref{ass:Linear}. Combining this with \eqref{eq:mui} and $\mu \leqslant 0$ proves $1^\top\lambda(t) \geqslant 0$.

First, we show that \eqref{eq:proofCostate2} preserves the ordering of $\lambda_i$, i.e., $\lambda_k(t_f) \geqslant \lambda_j(t_f)$ implies $\lambda_k(t) \geqslant \lambda_j(t)$ for all $t\leqslant t_f$. If there exists $t$ such that $\lambda_k(t) < \lambda_j(t)$, \eqref{eq:proofCostate2} shows  
\begin{equation}  \label{eq:proofCostate3}
\lambda_k(t_f) = e^{-a_k(t_f-t)}\lambda_k (t)+ \int_{t}^{t_f}e^{-a_k(t_f-\tau)}   p_k^\star[\tau] \psi(\tau)d\tau.
 \end{equation}
Assumption~\ref{ass:Linear}b and $\lambda_k(t_f) \geqslant \lambda_j(t_f)$ show $-a_k \leqslant -a_j$ and $p_k^\star[\tau] \leqslant p_j^\star[\tau] $ for a.e. $\tau$, leading to
\begin{align*}
 e^{-a_k(t_f-t)}\lambda_k (t) &<     e^{-a_j(t_f-t)}\lambda_j (t),\\
 e^{-a_k(t_f-\tau)}   p_k^\star[\tau] \psi(\tau) &\leqslant  e^{-a_j(t_f-\tau)}   p_j^\star[\tau] \psi(\tau),
\end{align*}
where the second inequality also uses $\psi(\tau) \geqslant 0$. The above shows $\lambda_k(t_f) < \lambda_j(t_f)$, reaching a contradiction. This proves that $\lambda_k(t_f) \geqslant \lambda_j(t_f)$ implies $\lambda_k(t) \geqslant \lambda_j(t)$ for all $t \leqslant t_f$. Following an analogous reasoning, $\lambda_{j_1}(t_f) > \lambda_{j_2}(t_f)$ implies that $\lambda_{j_1}(t) > \lambda_{j_2}(t)$ for all $t \leqslant t_f$.

Given the order-preserving property, we prove $1^\top \lambda(t)>0$ for any time interval. Suppose there exists $(t_1,t_2) \subseteq [0,t_f]$ where $1^\top \lambda=0$, then \eqref{eq:proofCostate} and $g^\top \lambda = 1^\top \lambda =0$ in this interval shows for any $t\in (t_1,t_2)$
\begin{equation}  \label{eq:proofNonzero}
0 = - 1^\top A \lambda(t) = -\sum_{i=1}^n a_i \lambda_i(t).
\end{equation}
However, Assumption~\ref{ass:Linear}b and the order-preserving property show that the two sequences $\{\lambda_i\}_{i=1}^n$ and $\{a_i\}_{i=1}^n$ satisfy the Chebyshev's sum inequality:
$$
\sum_{i=1}^n a_i \lambda_i(t) > \frac{1}{n} \Bigg(\sum_{i=1}^n a_i\Bigg) \Bigg(\sum_i^n\lambda_i(t)\Bigg) =0,
$$
which is a strict inequality due to $\lambda_{j_1} > \lambda_{j_2}$ and $a_{j_1}> a_{j_2}$. The above contradicts with \eqref{eq:proofNonzero}, proving $1^\top \lambda(t)\not=0$ for any time interval and thus $1^\top \lambda(t)>0$ for a.e. $t$.

\subsection{Proof of Proposition~\ref{prop:SpecialLinear}}
Analogous to the proof of Theorem~\ref{thm:linear}, \eqref{eq:proofCostate2} and $\psi(t) \geqslant 0$ hold, which leads to $p^\star_{k}[t]	\psi(t) \geqslant 0$ for a.e. $t$ based on Assumption~\ref{ass:SpecialLambda}c. The above and Assumption~\ref{ass:SpecialLambda}a show that $\lambda_j(t) \leqslant 0$ for all $t \leqslant t_f$ and all $j \in \{2,\dots,n\}$, and $\lambda_1(t) \geqslant 0$ for all $t \leqslant t_f$.

Then we prove that $g^\top \lambda(t)=1^\top \lambda(t)>0$ for a.e. $t$. Suppose there exists $(t_1,t_2) \subseteq [0,t_f]$ where $1^\top \lambda=0$. Then \eqref{eq:proofCostate} and $g^\top \lambda = 1^\top \lambda =0$ in this interval show that, for any $t\in (t_1,t_2)$,
\begin{equation}  \label{eq:proofNonzeroa}
0 = - 1^\top A \lambda(t) = -\sum_{i=1}^n a_i \lambda_i(t).
\end{equation}
Due to $1^\top \lambda =0$ and $\lambda \not=0$ from Proposition~\ref{prop:NonZeroLambda}, $\lambda_1(t)>0$ for $t\in (t_1,t_2)$, and $\lambda_j<0$ for some $j \in \{2,\dots,n\}$. $\lambda_j<0$ shows that  $p^\star[t]\not=0$, defined in \eqref{eq:proofCostate}, for a time interval after $t_2$, i.e., a constraint is active after $t_2$. Then, based on Assumption~\ref{ass:SpecialLambda}b, $\lambda_k(t)<0$ must hold for $t\in (t_1,t_2)$. Then $1^\top \lambda =0$ leads to
\begin{equation}  \label{eq:proofNonzero2}
\sum_{i=1}^n a_i \lambda_i(t)= \sum_{i=2}^n(a_i-a_1) \lambda_i>0,
\end{equation}
where the inequality holds because (i) for all $i\geqslant 2$, $\lambda_i \leqslant 0$ and $a_i-a_1\leqslant 0$ due to Assumption~\ref{ass:SpecialLambda}, and (ii) $\lambda_k <0$ and $a_1-a_k<0$ due to Assumption~\ref{ass:SpecialLambda}b. Thus, \eqref{eq:proofNonzero2} and \eqref{eq:proofNonzeroa} contradict, showing $1^\top \lambda(t)>0$ for a.e. $t \in [0,t_f]$.

\subsection{Voltage Function} \label{apex:Vol}
\begin{align*}
&V(c_{+,\mathrm{s}},c_{-,\mathrm{s}},I)  = \tau \mathrm{sinh}^{-1}\!\left(\frac{I}{221.69 i_0^+(c_{+,\mathrm{s}})} \right)\\& -\tau \mathrm{sinh}^{-1}\!\left(\frac{-I}{189.6681 i_0^-(c_{-,\mathrm{s}})} \right)\\&+ U_+\left(c_{+,s}(t)\right)- U_-\left(c_{-,s}(t) \right) +R_f I(t) \nonumber,
\end{align*}
where $\tau =1.236\times 10^{-8}$, $R_f=0.0022$, $U_+$, $U_-$ are in \cite{torchio2016lionsimba}, and 
\begin{align*}
 i_0^+ &= 5.031\!\times\! 10^{-11} \sqrt{1000 c_{+,\mathrm{s}}(51554 - c_{+,\mathrm{s}})}, \\
  i_0^- &= 2.334\!\times\! 10^{-11} \sqrt{1000 c_{-,\mathrm{s}}(30555 - c_{-,\mathrm{s}})}. 
\end{align*}

\end{document}